\documentclass[12pt,notitlepage]{article}% Math and Physical Sciences Reference Style
\usepackage{amssymb,amsmath, verbatim,
esint,amsthm,mathrsfs,bm,xcolor,mathrsfs,enumitem}
\usepackage[margin=1.25in]{geometry}
\usepackage{hyperref}
\hypersetup{colorlinks=true,linkcolor=blue,citecolor=blue, linktocpage}

\usepackage[sort, numbers]{natbib}
\usepackage{xcolor,colortbl}
\usepackage{graphicx}
\theoremstyle{theoremstyleone}%
\newtheorem{theorem}{Theorem}%  meant for continuous numbers
\newtheorem{prop}[theorem]{Proposition}% 
\newtheorem{lemma}[theorem]{Lemma}%
\newtheorem{cor}[theorem]{Corollary}%

\theoremstyle{theoremstyletwo}%

\theoremstyle{theoremstylethree}%
\newtheorem{defi}{Definition}%

\theoremstyle{definition}
\newtheorem{rem}[theorem]{Remark}

\newcommand{\R}{\bm{\mathbb{R}}}
\newcommand{\W}{\bm{\mathrm{W}}}
\newcommand{\z}{\bm{\mathrm{z}}}
\newcommand{\Sone}{\bm{\mathbb{S}}^{1}}

\newcommand{\T}{\mathbb{T}}
\newcommand{\Z}{\bm{\mathbb{Z}}}

\renewcommand{\div}{\mathrm{div}\,}
\newcommand{\supp}{{{\mathrm{spt}}\;}}

\begin{document}

\title{Dynamics of density patches in infinite Prandtl number convection}

%%=============================================================%%
%% Prefix	-> \pfx{Dr}
%% GivenName	-> \fnm{Joergen W.}
%% Particle	-> \spfx{van der} -> surname prefix
%% FamilyName	-> \sur{Ploeg}
%% Suffix	-> \sfx{IV}
%% NatureName	-> \tanm{Poet Laureate} -> Title after name
%% Degrees	-> \dgr{MSc, PhD}
%% \author*[1,2]{\pfx{Dr} \fnm{Joergen W.} \spfx{van der} \sur{Ploeg} \sfx{IV} \tanm{Poet Laureate} 
%%                 \dgr{MSc, PhD}}\email{iauthor@gmail.com}
%%=============================================================%%

\author{Hezekiah Grayer II}
\date{}

\maketitle
\begin{abstract}
This work examines the dynamics of density patches 
in the 2D zero-diffusivity Boussinesq system modified such that
momentum is in a large Prandtl number balance. We establish the global
well-posedness of this system for compactly supported and bounded
initial densities, and then examine the regularity of the
evolving boundary of patch solutions.
For $k \in \{0,1,2\}$, we prove the
global in time persistence  of
$C^{k+\mu}$-regularity, where $\mu \in (0,1)$, for the density patch
boundary via estimates of
singular integrals.
We conclude with a 
simulation of an initially circular
density patch via a  level-set method.  The simulated patch boundary forms corner-like structures with
growing curvature, and yet our analysis shows the 
curvature will be  bounded for all finite times. 
\end{abstract}
\section{Introduction}
In  Earth's mantle and many highly-pressurized gases, the rate of
thermal diffusion $\kappa$ is neglible compared to the rate of momentum
dissipation $\nu$. We will analyze the dynamics of density patches, modeling idealized plumes, in such fluids where the  nondimensional Prandtl number $\mathrm{Pr} = \nu / \kappa$ is large. 

Many scenarios of convection, including those which concern the present
study, are well-modeled by the Boussinesq system:
\begin{equation*}
\label{eq:B}
%(B)
\tag{$B$}
\left \{
  \begin{aligned}
    \left( \frac{\partial}{\partial t} + u\cdot\nabla\right) \theta &= \kappa\Delta
    \theta, \\
    \left( \frac{\partial}{\partial t} + u\cdot\nabla\right) u &= \nu\Delta u -
    \nabla \Pi  + \theta e_2, \\
    \mathrm{div}\, u &= 0,
  \end{aligned}
\right.
\end{equation*}
with initial temperature and velocity $(\theta_0,u_0): \R^2 \to \R \times \R^2$.  The unit
vector  $e_2 = (0,1)$ is the vertical direction
antiparallel to gravity, and the pressure $\Pi$ enforces
incompressibility at each instance of time.

For the system (\ref{eq:B}), global in time regularity has been well established for
fixed positive Prandtl number with $\kappa,\nu$ both positive (see e.g.
\cite{cannon1980initial}). The global regularity for $\kappa,\nu$ both zero
(degenerate $\mathrm{Pr}$) remains unknown.

Known is  the long time persistence of regularity for the partial viscosity scenarios---$\kappa$
 positive with zero $\nu$, and $\nu$ positive with zero
$\kappa$---given $(\theta_0,u_0) \in  H^{m}(\R^2)$ with integer $m > 2$
(\cite{Chae2005GlobalRF,Hou2004GLOBALWO}). For initial data of such regularity,
\cite{Chae2005GlobalRF} also provides the negative answer to XXI Century Problem 3 (see \cite{moffatt2001some})
concerning the development of $\nabla \theta$ singularities in the limit
$\kappa \to 0$ with fixed $\nu$ positive (formally $\mathrm{Pr} \to \infty$).

%In scenarios of free convection, such limiting dynamics is made
%apparent by the non-dimensional Boussinesq system
%in the thermal
%diffusive time scale. We have the evolution of temperature, 
%\begin{equation*}
%  \left( \frac{\partial}{\partial t^{*}} + u\cdot\nabla\right) \theta = \Delta \theta,
%\label{thetaeqwL}
%\end{equation*}
%and the balance of momentum,

%Notice here $t^{*} = \kappa t /L^2$ where $t$ is the dimensional
%time and  $L$ is the length scale.%Here, the Prandtl number $\mathrm{Pr} = \nu / \kappa$ quantifies the relative
%dominance of the nonlinear terms $u \cdot  \nabla \theta$ and $u \cdot
%\nabla u$ in the dynamics of $(\theta,u)$.

The question of
well-posedness and global in time  regularity for less regular initial data has been actively studied (e.g.
\cite{Abidi2007OnTG,hmidi2007global,danchin2008theoremes}) 
for the zero diffusivity system 
\begin{equation}
  \tag{$B_1$}
\label{eq:B1}
\{\text{system (\ref{eq:B}) with }\kappa = 0 \text{ and }\nu\text{ positive}\}.
\end{equation}
Without thermal diffusivity, convective plumes are idealized as density distributions initially of the form
\begin{equation}
\label{eq:init}
\theta_0 = \bm{1} _{P_0}, 
\end{equation}
where $P_0 \subset \R^2$ is simply connected and bounded. If $u_0$ is
sufficiently regular, then the regularity
results in \cite{Chae2005GlobalRF} guarantee a unique solution 
to system (\ref{eq:B1}) of the form
\begin{equation}
\label{eq:patch_t}
\theta(t) = \bm{1}_{P(t)},
\end{equation}
where $u(t)$  is regular enough to define the flow map 
$X(\cdot,t) : \R^2 \mapsto \R^2$
such that the evolving region
$P(t) = X(P_0,t)$
remains simply connected and bounded. However, inferring the
regularity of the evolving boundary $\partial P(t)$ in time is a
classical problem. 
Given $k \in \Z_{+} = \{1,2,\ldots\} $ and $\mu \in
(0,1)$, we write
  $\partial P(t) \in C^{k+\mu}$ if there exists some  
 $\z : \Sone \to \R^2$ such that 
\begin{equation*}
\partial P(t) = \{\z(\alpha) \mid \alpha
\in \Sone \} \quad \text{and} \quad \z \in C^{k+\mu}(\Sone).     
\end{equation*}
%Henceforth, we say $\partial P \in
%C^{k, \mu}$ if there exists a parametrization $z$ such that $
%\partial P = \{z(\alpha), \alpha \in \Sone \} $ with $z(\cdot ) \in C^{k,\mu}$

The question of persistence of regularity for the patch boundary is
precisely whether  $\partial P(t) \in C^{k+\mu}$ given  $\partial P(0) =
\partial P_0 \in C^{k+\mu}$. This line of inquiry has its origin in the
study of vortex patches in the two-dimensional incompressible Euler
system. At some point, numerical studies like \cite{buttke1989observation,dritschel1990does} were done which had conflicting conclusions on
whether finite time contour singularities formed in the evolution of a vortex patch. Then, \cite{chemin1993persistance} showed global in time persistence of smoothness using paradifferential calculus and
striated regularity methods to obtain the desired regularity of $\nabla u$; soon after, \cite{Bertozzi1993GlobalRF}
proved a similar result using geometric analysis techniques: the vortex patch boundary remains $C^{1+\mu}$ for $\mu
\in (0,1)$, given it was so initially. Another proof of this was
given by \cite{serfati1994preuve}.

More recently for the system  (\ref{eq:B1}), the global persistence of $C^{1 +
\mu}$-regularity
(\cite{Danchin2016GlobalPO}), $C^{2 + \mu}$-regularity
(\cite{gancedo2017global}), and $C^{k+\mu}$-regularity for all $k
\in \Z_{+} $
(\cite{chae2021global}) was
shown for density patches. The $k = 1$ and  $k \in \Z_{+}$ results made special use of striated
regularity estimates in Besov spaces.
%\note{Discuss historical tension between numerical simulations of free
%boundary problems and analytical results.}
%The 
%formation of countour singularities in the evolving boundary $\partial
%P(t)$ is of particular concern, regularity is known to break down in
%other contexts  .

%In particular, the development of
%singularities
The present study connects this inquiry of contour singularity formation
to the program of XXI Century Problem 3. In the limiting dynamics of
large Prandtl number, the fluid velocity is at relative equilibrium in
the thermal diffusive time scale $\tau_\kappa = L^2/\kappa$, where $L$
is the length scale. The nondimensional momentum equation is
\begin{equation}
  \begin{aligned}
  \frac{1}{\mathrm{Pr}} \left(\frac{\partial}{\partial t^{*}} + \underline{u} \cdot
    \nabla^{*} \right)
    \underline{u} &= \Delta^{*} \underline{u} - \nabla^{*}\underline{\Pi}  + \mathrm{Ra} \, \underline{\theta} e_2, \\ 
  \end{aligned}
\label{uPr}
\tag{$*$}
\end{equation}
where the starred derivatives have been rescaled by $\tau_\kappa$ and
$L$, and the variables with underbars are the nondimensional
counterparts to $(u,\theta,\Pi)$. The Rayleigh number $\mathrm{Ra}$ is
independent to $\mathrm{Pr}$. Thus formally the material derivative
in  (\ref{uPr}) vanishes in the limit of large Prandtl number;
specifically, in thermal diffusive time $t^{*} = t/\tau_\kappa$, the
fluid is in equilibrium dominated by viscosity.
Accordingly, we consider the (dimensional) zero diffusivity system
(\ref{eq:B1}) with the modified momentum equation
\begin{equation}
\label{eq:u}
 - \Delta u + \nabla\Pi  = \theta e_2. \\ 
%\label{uPreq}
%\tag{$*$}
\end{equation} 

%The present study concerns a simplified $(B_1)$ where we neglect the
%fluid inertia according to the large Prandtl
%number limit of the  non-dimensional equation (\ref{uPr}).
We thus consider patch dynamics in the system given by 
%\begin{equation}
%  \left(\frac{\partial}{\partial t} + u\cdot\nabla\right) \theta = 0
%\label{thetaeq}
%\end{equation}
%with 
%\begin{equation}
%-\Delta u + \nabla\Pi  = \theta e_2; \quad \mathrm{div}\, u = 0.
%\label{u}
%\end{equation}
\begin{equation*}
\label{eq:Bs}
%\tag{$B_{*}$}
\tag{$B_{*}$}
\left \{
  \begin{aligned}
    \left( \frac{\partial}{\partial t} + u\cdot\nabla\right) \theta &= 0, \\
    -\Delta u +
    \nabla \Pi  &= \theta e_2, \\
    \mathrm{div}\, u &= 0,
  \end{aligned}
\right.
\end{equation*}
with initial data $\theta_0: \R^2 \to \R$. Above, the solution
$\theta(x,t)$ is an active scalar function of  $x\in \R^2$ and $t\in
\R$, where the corresponding velocity $u(t) : \R^2 \to \R^2$ is given by the balance (\ref{eq:u}) at each instance of time.

The system \eqref{eq:Bs} is known also as transport-Stokes, and can
model the sedimentation  of inertialess particles in Stokes flow
\cite{hofer2018sedimentation,mecherbet2018sedimentation} . Let us
briefly recount the literature concerning global well-posedness of this
system in
unbounded domains. Global well-posedness  for $L^{\infty}$
initial data holds on the infinite strip $ \R \times (0,1)$ with no-slip boundary conditions and a flux condition
 imposed
\cite{leblond2022well}.
On the whole space $\R^{3}$, the analogue of system
\eqref{eq:Bs} is globally well-posed for initial data in $L^{1} \cap
L^{\infty}$ \cite{hofer2021influence}; however, whether this is true in
$2d$ without additional assumptions is unknown.
%The main issue is the
%asympotic behavior of the fluid velocity as $|x| \to \infty$, which is
%decay for $d \geq 3$, but is like $\log{|x|}$ when $d = 2$.
Recent discussions of the
 well-posedness problem may be found in
\cite{mecherbet2022few,cobb2023well}.

The  question of regularity for  an interface of density in the
transport-Stokes system was first addressed in \cite{antontsev2000free}.
On bounded domains in $\R^{d}$ with $d \geq 2$, \cite{antontsev2000free}
establishes  global well-posedness 
for piecewise constant initial densities, where the boundary conditions
are no-slip, and then shows that interfaces which
are initially $C^{1 + \mu}$, will remain so for all
time. More recently, \cite{gancedo2022long} considers a contour dynamics equation derived from system
\eqref{eq:Bs} on the horizontally periodic strip $\T \times \R$, wherein interfaces are shown to be locally
well-posed in $C^{1 + \mu}$, even in the Rayleigh-Taylor unstable regime.

%The recent work \cite{mecherbet_2021} addresses this question when
%the domain is $\R^{3}$, and shows the persistence of  Lipschitz and
%$C^{1+\mu}$ regularity for the evolving surface  enclosing an evolving bounded
%region of uniform density.

Our contributions concern the system  \eqref{eq:Bs} on the whole space
$\R^2$. We first establish the global well-posedness of the system for
initial data of
Yudovich-type, meaning when  $\theta_0 \in   L^{1} \cap L^{\infty}$ is
compactly supported. In particular, initial
data (\ref{eq:init}) admit patch solutions \eqref{eq:patch_t}
and we may study the regularity of  $\partial
P(t)$ in time. Then, we prove the global 
persistence of $C^{k+\mu}$ regularity for $\partial P$  when $k
  \in \{0,1,2\} $.
  Further, we simulate on $\T^2$ the dynamics of an initially circular
  patch, and observe the interface develop corner-like structures with
  growing curvature. The results here rule out the 
  possibility of a finite-time curvature singularity; however, corner formation at infinite time remains 
possible.

%the formation of contour singularities in the inifinite Prandtl number limit.
%
%\note{Summary paragraph of purpose, goals and context of paper before going somewhat into details below}

The content is organized as follows. The main contributions are
presented in Section \ref{sec:main}. The key Lemmas \ref{lem:pointu}
and \ref{lem:key} are proved, giving the details of the relevant singular
integrals, in Section \ref{sec:key}. The global well-posedness of the system
(\ref{eq:Bs}), in particular Theorem \ref{thm:well}  concerning the
Yudovich class,  is addressed in Section \ref{sec:well}. 
In Section \ref{sec:reg}, the main Theorem \ref{thm:reg} on the persistence of regularity for the
evolving boundary $\partial P(t)$ is finally proved. 
In Section \ref{sec:fig}, the numerical simulation in Figure
\ref{fig:circ} is described in detail.

\section{Main Results}\label{sec:main}
We suppose an initial distribution 
 (\ref{eq:init}) in the Cauchy problem for system (\ref{eq:Bs}) and show
 the ensuing solution velocity $u(t)$ is regular enough such that the
 solution temperature has patch representation \eqref{eq:patch_t},
where the evolving region 
\begin{equation}
\label{eq:}
P(t) = X(P_0,t)  
\end{equation}
is given by the flow map
$X(\cdot,t)$ of our fluid velocity. 

We start by introducing the
streamfunction $\psi$ such that $u = \nabla^{\perp} \psi$
with $\nabla^{\perp} = (-\partial_2,\partial_1)$, and then the
momentum equation in (\ref{eq:Bs}) yields the following expression for vorticity
$\omega \equiv \nabla^{\perp} \cdot u$ at each instance of time:
\begin{equation}
\label{eq:omethet}
  -\Delta \omega = \partial_1 \theta.
\end{equation}

The fundamental solution of the bilaplacian $\Delta^2$ in $\R^2$ is
\begin{equation}
\frac{1}{8\pi} |z|^2(\log|z| - 1)
\end{equation}
where $z \in \R^2$. Given the fact $\Delta \psi = \omega$ generally, and
assuming $\theta$ in equation (\ref{eq:omethet}) has
enough decay at infinity, it follows
$\psi = -(\Delta^{2})^{-1}\partial_1\theta$ such that $\psi = K * \theta$ with kernel
%\begin{equation}
%%\Psi(x) = \int_{\R^2} K(x - y) \theta(y) \, dy
%\end{equation}
%with
\begin{equation}
K(z) = -\frac{z_1}{4\pi}\left(\log{|z|}-\frac{1}{2}\right).
\label{K}
\end{equation}
More precisely, we find the expression for vorticity
\begin{equation}
\omega(x)  =  -\frac{1}{2\pi}\int_{\R^2}\frac{x_1-y_1}{|x-y|^2}
\theta(y) \, dy,
\label{omegat}
\end{equation}
and note that the integral converges absolutely for $\theta$ in $Y \equiv L^1(\R^2)\cap
L^{\infty}(\R^2)$. Denoting the integral as $\omega = Q *
\theta$, this is a convolution with a singular kernel which is homogeneous of
degree $-1$.
%, hence bounded from $L^p$ to $L^q$ for $1 < p < q < \infty$ with $1/q = 1/p - 1/2$. 
Thus, the operator $(Q*\vphantom{k}): Y \to L^p$ giving vorticity  from
temperature is bounded for $2 < p \leq \infty$; we give a direct proof
of this claim in Section \ref{sec:key}.

Let $G (z)\equiv \nabla^\perp K(z)$, then the integral $u = G * \theta$, explicitly
\begin{equation}
u(x) = \int_{\R^2}G(x-y)\theta(y) \, dy,
\label{UintT}
\end{equation}
converges absolutely if $\theta \in Y$ is compactly supported. More
quantitatively:
\begin{lemma}
  \label{lem:pointu}
  Suppose $\theta \in L^{\infty} $ with $ \supp \theta \subset B(0,R)$.
  Let $u = G  * \theta$. Then,
  \begin{equation*}
  \label{eq:uLpT}
  | u(x) |  \leq C (1 + \log (|x| + R + 1)) \|\theta \|_{Y}.
  \end{equation*}
\end{lemma}

From here, we find the gradient as $\nabla u = \nabla G *
\theta$, and further, we compute the principal value of the second gradient of velocity $\nabla \nabla u(x)$. With the kernel $\nabla \nabla G  : \R^2 \to \R^{2\times 2 \times 2}$, we have
\begin{equation}
\label{eq:gradUkernel}
\nabla  \nabla  u(x)  =  \theta(x) \bm{\mathrm{E}} + \frac{1}{4 \pi} \mathrm{pv}
\int_{\R^2} \theta(x - z ) \nabla \nabla G(z)
 \, dz.
\end{equation}
The entries of the $\bm{\mathrm{E}} \in \R^{2\times 2
\times 2}$, as well as $G$ and its gradients, $\nabla G$ and
$\nabla \nabla G$, are given in 
Section \ref{sec:key}. Of special note here is that the
integral operator given by kernel $\nabla \nabla G$  is 
Calder\'{o}n-Zygmund in each entry, thus bounded on $L^p$ for $ p \in
(1, \infty)$. %Overall, the action of (\ref{eq:gradUkernel}) giving $\nabla \nabla u$ from $\theta$ is bounded from $Y$ to $L^p$ for $1< p < \infty$, similar to the operator $(I + T)$. 
Combining this with the previous observations, we deduce:
\begin{lemma}
  \label{lem:key}
  Suppose $\theta\in Y$. Let $
  \nabla u =  \nabla G
  * \theta$. Then for $\mu \in (0,1)$, there exists constant $C_{\mu}$ 
  depending only on $\mu$ such that
  \begin{equation*}
      \|\nabla u\|_{C^{\mu}} \leq C_{\mu} \|\theta\|_{Y}.
  \end{equation*}
\end{lemma}

Initial data which satisfy the hypotheses of both Lemmas
\ref{lem:pointu} and \ref{lem:key}, namely
compactly supported and bounded functions, are called Yudovich-type due
to the classical result \cite{yudovich1963non} establishing the global
well-posedness of weak solutions within this class for the
 two-dimensional Euler system. This result allowed the study
of long-time dynamics for vortex patches, in particular the program of
vortex patch boundary regularity.  For the system (\ref{eq:Bs}), the
initial data of our concern are clearly Yudovich-type,
and the Lemmas are powerful enough to establish the following: 
\begin{theorem}
  \label{thm:well}
  Let $\theta_0\in L^1 \cap L^\infty$ and $\supp \theta_0 \subset B(0,R_0)$. Then for
  arbitrary $T$,
  the system (\ref{eq:Bs}) has a unique weak  solution
  $$\theta \in L^{\infty}(0,T; L^1 \cap L^\infty ).$$
  Moreover, for $t \leq T$, the
  following estimates hold:

  \begin{enumerate}
    \item For $1 \leq p \leq \infty$,
      \begin{equation*}
      \label{eq:LpT}
      \|\theta(t)\|_{L^{p}} \leq \|\theta_0\|_{L^p}
      \end{equation*}

    \item For $R(0) = R_0$, we have
      \begin{equation*}
      \label{eq:suppT}
      \supp \theta(t) \subset B(0,R(t))
      \end{equation*}
      with $R(t)$ obeying the differential inequality
       \begin{equation*}
      \label{eq:dRT}
      \frac{dR(t)}{dt} \leq C (1 + \log(2R(t)+1)) \|\theta_0\|_{Y}
      \end{equation*}

    \item For $\mu \in (0,1)$, 
       \begin{equation*}
       \label{eq:cUT}
       \|\nabla u(t)\|_{C^{\mu}} \leq C_\mu \|\theta_0\|_{Y}
       \end{equation*}
       where $u(t) = G * \theta(t)$.
  \end{enumerate}
\end{theorem}
\begin{cor}
  \label{cor:patch}
  Let $\theta_0 = \bm{1}_{P_0}$ where $P_0 \subset \R^2$ is simply connected and
  bounded. Then, the unique
  weak solution $\theta$ to system  (\ref{eq:Bs})  has the form 
  \begin{equation*}
    \theta(t) = \bm{1}_{P(t)}
  \end{equation*}
  where $P(0) = P_0$ and $P(t)$ is simply connected and bounded.
\end{cor}

The corollary is due to an explicit construction: we represent the dynamics of the distribution
$\bm{1}_{P(t)}$ via the evolution of a smooth, compactly supported level-set function 
$\varphi(t) \in C_c^{\infty}$. Specifically, we suppose  $\varphi_0 \in C_c^{\infty}$ such
that 
\begin{equation}
\label{eq:patchphi0}
P_0 = \{ x  \mid \varphi_0(x) > 0 \}\quad \text{and} \quad \partial P_0
= \{x  \mid
\varphi_0(x) = 0 \},
\end{equation}
for $P_0$ in the claim. Then, we let $\varphi$ be transported passively by $u(t) =
G * \theta(t)$ where $\theta(t) \in Y$ is the unique weak solution with
initial data
$\bm{1}_{P_0}$ guaranteed to exist by Theorem \ref{thm:well}. Since
$u(t)$ is locally bounded and $\nabla u(t) \in C^{\mu}$ for all time, we have
$\varphi$ as the unique global solution the
Cauchy problem given by
\begin{equation}
\label{eq:cauchyphi}
\left(\frac{\partial}{\partial t} + u \cdot \nabla \right) \varphi = 0,
\end{equation}
with initial data $\varphi_0$. Further,  we have explicitly $\varphi(t)
= \varphi_0 \circ X^{-1}(\cdot,t)$ where $X$ is the well-defined flow map
given by $u$. It follows that the patch representation $\theta(t) = \bm{1}_{P(t)}$ holds
 for all time, where $P(t) =
X(P_0,t)$ is defined by
\begin{equation}
\label{eq:patchphi}
P(t) = \{ x  \mid \varphi(x,t) > 0 \}\quad \text{and} \quad \partial P(t)
= \{x  \mid
\varphi(x,t) = 0 \}.
\end{equation}

With global  well-posedness of patches solutions to
 system (\ref{eq:Bs}) established, we may now  examine the question of
global in time regularity for the evolving boundary $\partial P(t)$. We
show first the local propogation of H{\"o}lder continuity:

\begin{cor}
  \label{thm:reg1}
Suppose that $ \partial P_0$ is $\mu$-H{\"o}lder continuous at
$\tilde{x}$ with  $\mu \in (0,1)$. Then, $\partial P(t)$ is
$\mu$-H{\"o}lder continuous at $X(\tilde{x},t)$ at any given time.
\end{cor}
This result follows from a Lagrangian
construction of
$\partial P(t).$ Rather than the construction (\ref{eq:patchphi}), we consider
some initial parametrization $\z_0 :  \Sone \mapsto \partial P_0$. We observe that the flow map $X(a,t)$ for fluid
velocity $u = G * \bm{1}_{P}$ then gives us the
evolving patch boundary
 \begin{equation}
\label{eq:parpar}
\partial P(t) = \{X(\z_0(\alpha),t)  \mid  \alpha \in  \Sone \},
\end{equation}
such that
\begin{equation}
\label{eq:langz}
\z(t) = X(\cdot,t) \circ \z_0  
\end{equation}
parametrizes $\partial P(t)$ for any and all time. We note that for all times the image of $\z(t)$ must be in
\begin{equation}
  B(0,R(t)) = \{|x| \leq R(t) \}
\end{equation}
wherein  $u(t)$ is $C^{1+\mu}(B(0,R(t)))$, by Theorem \ref{thm:well}. Thus if $\z_0$ is $\mu$-H{\"o}lder continuous at $\tilde{\alpha}$, then the composition of maps giving $\z(t)$ is $\mu$-H{\"o}lder continuous
at $\tilde{\alpha}$. Arranging $\z_0(\tilde{\alpha}) = \tilde{x}$, we have established Corollary \ref{thm:reg1}.

Let us briefly demonstrate the difficulty in proving the persistence of higher
regularity in the Lagrangian construction by examining the contour dynamics equation (CDE) for system
(\ref{eq:Bs}), which we proceed to derive. 

By its construction (\ref{eq:langz}), $\z(t)$ 
obeys the evolution
equation
\begin{equation}
\frac{d\z}{dt} = u(\z(t),t).
\end{equation}
Suppose  $\z_0$ is
a parametrization by arc length, such that $\z(t)$ is also. 
Examining the expression for  velocity $u(t) = \nabla^{\perp}K *
\bm{1}_{P(t)}$, we apply Green's theorem to discover
\begin{equation}
\begin{aligned}
  u(x,t) &= \int_{P(t)} \nabla^{\perp} K(x - y)\, dy \\
   &= -\int_{0}^{2\pi} K(x - \z(\sigma,t)) \frac{\partial \z}{\partial
\alpha}(\sigma,t) \, d\sigma,
\end{aligned}
\end{equation}
where $\z(t)$ is assumed to be clockwise oriented. Suppressing the time
argument, we now have the contour dynamics equation
\begin{equation}
  \tag{CDE}
\label{eq:CDE}
\frac{\partial \z}{\partial t}(\alpha) = -\int_{0}^{2\pi} K(\z(\alpha) - \z(\sigma)) \frac{\partial \z}{\partial
\alpha}(\sigma) \, d\sigma.
\end{equation}
%This integral converges absolutely given $\partial_{\alpha}\z$ is
%bounded, or more generally if $\partial_{\alpha}\z_{j}(\sigma)$ grows at
%most like $ 1/ |\alpha
%- \sigma|$ in
%neighborhoods of $\alpha$. 
So, we differentiate the above
equation in $\alpha$ to find the evolution equation
\begin{equation}
\label{eq:langdz}
\frac{\partial^2 \z}{\partial t \partial \alpha}(\alpha) = -\int_{0}^{2\pi} \left
\lbrack \nabla K(\z(\alpha) - \z(\sigma)) \cdot
\frac{\partial\z}{\partial \alpha}(\alpha) \right\rbrack \frac{\partial \z}{\partial
\alpha}(\sigma) \, d\sigma,
\end{equation}
where explicitly
\begin{equation}
\nabla K(z) = \frac{1}{8\pi} \left\lbrack 
\begin{pmatrix} 
  -2 \log|z| \\
  0
\end{pmatrix} -\frac{1}{|z|^2}
\begin{pmatrix} 
z_{1}^2 - z_{2}^2 \\
2 z_1 z_2
\end{pmatrix} 
\right\rbrack .
\end{equation}
From here, we see that proving the global in time well-posedness of the \ref{eq:CDE} by itself is
difficult, even more so is  showing the global in time persistence of $C^{\mu}$
regularity for  $\partial_{\alpha} \z(t)$ evolving via (\ref{eq:langdz}).
This challenge is typical of contour dynamics models.

\begin{theorem}
  \label{thm:reg}
Let $k \in \{0,1,2\}$ and suppose that $ \partial P_0 \in
  C^{k+\mu}$ with  $\mu \in (0,1)$. Then, $\partial P(t) \in C^{k + \mu}$ 
  for all time.
\end{theorem}

The higher regularity results follow from an analysis of the Eulerian
construction of $\partial P$ in Corollary \ref{cor:patch}. Recalling the
level set function $\varphi$,  we note the direction of the vector field $\W = \nabla^{\perp} \varphi$ is
tangent to $\partial P = \partial P(t)$ in general. We then  have the parametrization
\begin{equation}
\label{eq:parah}
\z : \Sone \mapsto \partial P \quad \text{with} \quad
\frac{\partial\z}{\partial \alpha} = \W\circ\z,
\end{equation}
given $\W$ is bounded and non-vanishing everywhere on $\partial P$. Thus to guarantee $\partial
P \in  C^{1+ \mu}$ for $\mu \in (0,1)$, the relevant quantities to
control are $\|\W(t)\|_{L^{\infty}}$ and  
  \begin{align}
    \lvert \W \rvert_{\inf} &:= \inf\limits_{x \in \partial P } \lvert\W(x)\rvert = \inf_{x,
    \varphi(x) = 0} \lvert \W(x) \rvert, \label{eq:quant1} \\
      \lvert\W \rvert_{\mu} &:= \sup_{x \neq x'}  \frac{\lvert \W(x) - \W(x') \rvert}{\lvert x -
    x'\rvert^{\mu}}, \quad \mu \in (0,1). \label{eq:quant2}
  \end{align}

Differentiating (\ref{eq:cauchyphi}), the evolution of $\W(t)$ where $\W_0 =
\nabla^{\perp}\varphi_0$ obeys 
\begin{equation}
\label{eq:weq}
\left( \frac{\partial}{\partial t} + u \cdot \nabla \right) \W = \nabla
  u \,  \W.
\end{equation}
As transport preserves $\| \theta(t)\|_{L^{p}}$,
the estimate of Lemma \ref{lem:key} gives immediately $\|\nabla u(t)\|_{L^{\infty}}$ is
bounded uniformly in time by a positive constant $C_L$ depending only on the initial
data. This is already sufficient to achieve $\partial P \in C^{1 + \mu}$ with an
$\exp(C_L \exp(C_L|t|))$ bound for  $|\W(t)|_{\mu}$ from  Gr\"{o}nwall-type
inequalities, see \cite[Proposition 3]{Bertozzi1993GlobalRF}. 

With the
geometrical insights of \cite{Bertozzi1993GlobalRF}, we may improve this
initial bound for $|\W(t)|_{\mu}$ to an $\exp(C_L |t|)$ bound.
Further, these geometrical methods are used to prove persistence of
$C^{2+\mu}$-regularity. Differentiating
(\ref{eq:weq}) yields an evolution equation,
\begin{equation}
\label{eq:dynDW}
\left(\frac{\partial}{\partial t} + u \cdot \nabla \right)\nabla  \W  
= \left\lbrack \nabla \W,  \nabla u \right\rbrack + \nabla \nabla u
\cdot \W ,
\end{equation}
where  we have the tensor product commutator $\lbrack A,B \rbrack =
A\cdot B - B\cdot A$. The term to control is the product $\nabla \nabla
u \cdot \W$, therein we observe the expression of $\nabla\nabla u$
for the patch $\theta = \bm{1}_{P}$ is
\begin{equation}
\nabla  \nabla  u(x)  =  \bm{1}_{P}(x) \bm{\mathrm{E}} + \frac{1}{4 \pi} \mathrm{pv}
\int_{\R^2} \bm{1}_{P}(x - z ) \nabla \nabla G(z)
 \, dz.
\end{equation}
Clearly estimating $|\nabla\nabla u|_{\mu}$ from this expression is
difficult, but we 
use the fact that the vector field $\W \in C^{\mu}(\R^2,\R^2)$ is divergence-free and
tangent to $\partial P$ to reduce the problem to estimating $\| \nabla\nabla
u\|_{L^{\infty}}$ (Corollary \ref{cor:com2}). 

Beyond satisfying the cancellation
property, the CZ kernel $\nabla \nabla G(z)$ has reflection symmetry
such that we have
$\bm{\mathrm{H}}:\Sone \to \R^{2 \times 2 \times 2}$ and
\begin{equation}
\nabla \nabla G(z) = \frac{\bm{\mathrm{H}}(z)}{|z|^2} \quad \text{with}
\quad \bm{\mathrm{H}}(-z) = \bm{\mathrm{H}}(z).
\end{equation}
Because small neighborhoods of $\partial P$ look like half-circles, the
reflection symmetry allows us to achieve a sufficient $L^{\infty}$ estimate on
$\nabla \nabla u(t)$  in Proposition \ref{prop:DDu}. Ultimately, we show $|\nabla\W(t)|_{\mu}$ may
grow like $|t|\exp(C_{3}|t|)$, where $C_3$ depends only on  $\mu$ and
the initial data, and the global in time persistence of
$C^{2+\mu}$ regularity for $\partial P(t)$ is proved. The complete
details are given in Section \ref{sec:reg}.

While the simulations of \cite{buttke1989observation} and
\cite{dritschel1990does} were initialized with two circular vortex
patches of positive sign, a single circular
patch of density is already unsteady in system \eqref{eq:Bs}. Observe
that for the initial data $\theta_0 = \bm{1}_{B(0,1)}$, the initial vorticity
$\omega_0$ evaluated on the contour $\z_0(\alpha) = (\sin \alpha, \cos
\alpha)$,
%\begin{equation}
%\label{eq:}
%\omega_{0}(\z_0(\alpha)) = \int_0^{2\pi} \int_0^{1} 
%\bm{1}_{B(0, 1)}\begin{pmatrix} \sin \alpha - r \cos \theta \\ \cos
%\alpha - r \sin \theta \end{pmatrix}  \cos \theta \, dr \, d \theta
%\end{equation}
\begin{equation}
\label{eq:}
\omega_{0} \circ \z_0 = \int_0^{2\pi} \int_0^{1} 
H(2r \sin(\alpha + \theta) - r^2)  \cos \theta \, dr \, d \theta
\end{equation}
where $H$ is the Heaviside function, is not symmetric in $\alpha$. Thus
the initially circular patch will not remain circular. More generally, for a nonnegative and nontrivial $\theta_0$,
the center of mass 
\begin{equation}
\label{eq:}
q = \frac{1}{M}\int_{\R^2}  \theta x\, dx, 
\end{equation}
where $M = \| \theta_0\|_{L^{1}}$, drifts
upwards with vertical component $q_2$ increasing as 
\begin{equation}
\label{eq:}
\frac{dq_2}{dt} = -\frac{1}{M}\int_{\R^2} 
\div(\theta u) x_2  \, dx = \frac{1}{M}\int_{\R^2}  \theta u_{2}  \, dx > 0
\end{equation}
because $u_2 = G_2 * \theta$ is
given by a positive operator on $\theta$. If the initial data has
a horizontal reflection symmetry, the horizontal component $q_1$ is
stationary.

Let us consider the system (\ref{eq:Bs}) on $\T^2$ such that the spatial domain is
 now computationally realizable. Observing that the
problem is Galilean invariant, we choose a reference frame so that
 the mean velocity $\int_{\T^2} u \, dx = \widehat{u}(0)$ is zero.
Then from integrating the momentum equation, we see that $\widehat{u}(0)$
is indeed constant in time, and that we must have $\int_{\T^2}
\theta\, dx = 0$ for compatibility. 
Clearly, the arguments giving the main results, in particular Lemmas
\ref{lem:pointu} and \ref{lem:key}, hold
readily on the domain $\T^2$ wherein $L^1 \cap L^{\infty} \equiv
L^{\infty}$. 
\begin{cor}
  \label{cor:welltor}
  Let $\theta_0 \in L^{\infty}(\T^2)$ such that
$$\int_{\T^2} \theta_0 \, dx = 0. $$
Then, Theorem \ref{thm:well} holds for the system \eqref{eq:Bs} on $\T^2$.
\end{cor}

Our definition of patches must be modified to be mean zero in
density. On the flat torus, the compatible patch data is
\begin{equation}
\label{eq:torinit}
  \theta_0 = \bm{1}_{P_0} - \mathrm{area}(P_0) 
\end{equation}
  where $P_0 \subset  \T^2$ is simply connected. For such initial data, the unique weak solution $\theta$ to
  system (\ref{eq:Bs}) on $\T^2$ then has the form 
\begin{equation}
\label{eq:thetator}
\theta(t) = \bm{1}_{P(t)} - \mathrm{area}(P_0),
\end{equation}
where $P(0) = P_0$ and $P(t)$ is simply connected. 
On this numerically
tractable domain, we here implement a level-set
method that approximates the evolution of the boundary $\partial P(t)$
which we know is well-defined for all time.
\begin{cor}
  \label{cor:patchtor}
  For the system  \eqref{eq:Bs} on $\T^2$, if the initial data $\theta_0$ has the form
  \eqref{eq:torinit}, then the unique weak solution $\theta$ in Corollary
  \ref{cor:welltor} has the form \eqref{eq:thetator}.
\end{cor}

\begin{figure*}
  \begin{center}
    \setlength\tabcolsep{-20pt}
    \setlength{\fboxrule}{0pt}
        \makebox[0pt]{
    \begin{tabular}{@{}ccccc@{}}
      $\partial P(t = 0)$ & $t = 25$ & $t = 50$ & $t = 75$ & $t = 100$\\  
      \raisebox{-.45\height}{\fbox{\includegraphics[scale=.55]
          {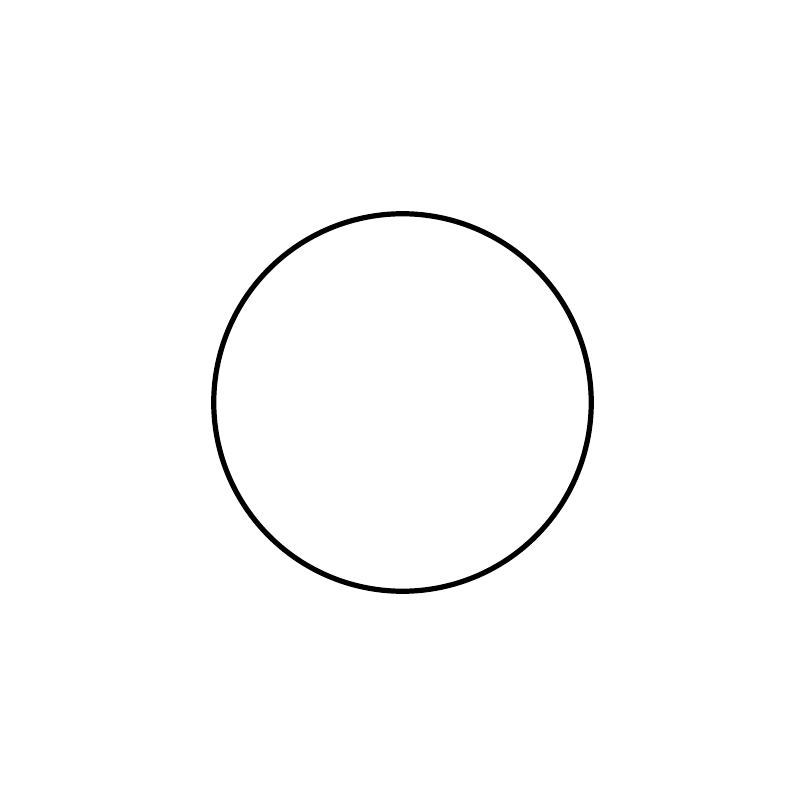}}} &
      \raisebox{-.45\height}{\fbox{\includegraphics[scale=.55]
          {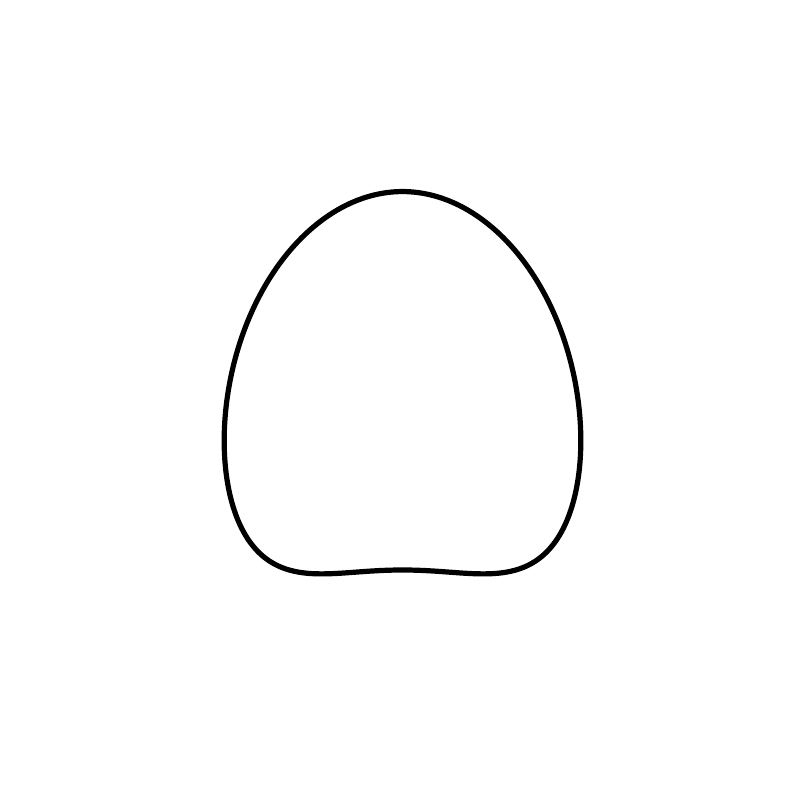}}} &
      \raisebox{-.45\height}{\fbox{\includegraphics[scale=.55]
          {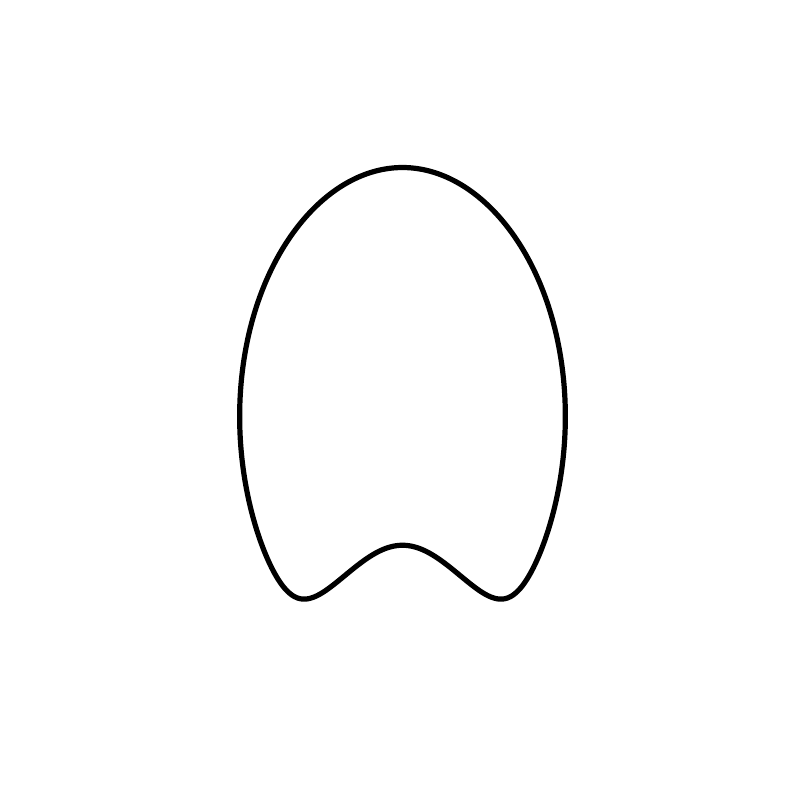}}} &
      \raisebox{-.45\height}{\fbox{\includegraphics[scale=.55]
          {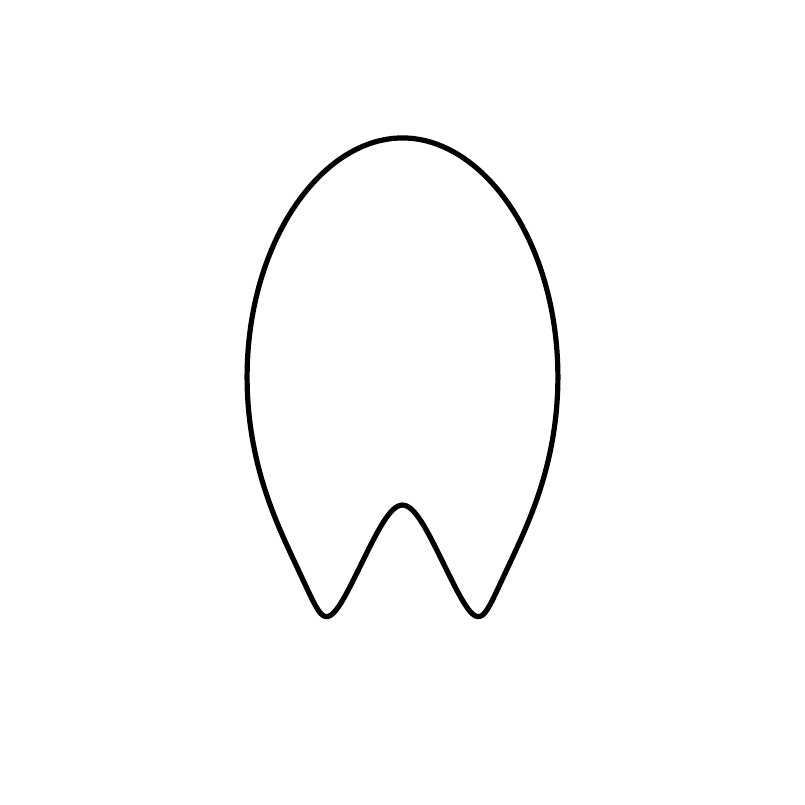}}} &
      \raisebox{-.45\height}{\fbox{\includegraphics[scale=.55]
          {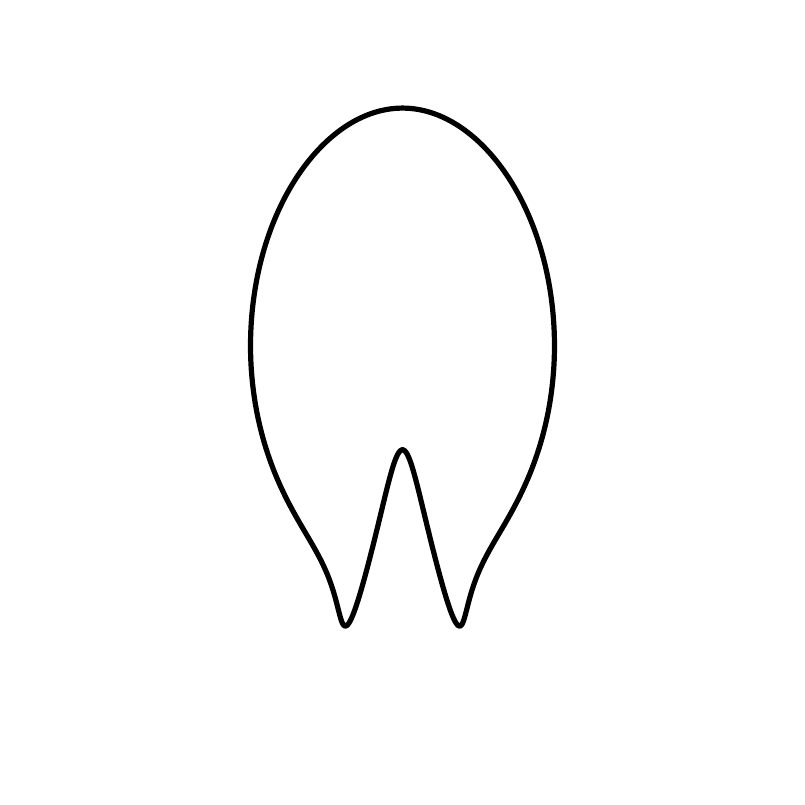}}} 
      \end{tabular}}
      \caption{Snapshots of $\partial P$ from a numerical patch solution ($N =
        1024$) to
      system (\ref{eq:Bs}) on $\T^2$ for
       $\partial P_0 = 
      \Sone(\frac{1}{2} )$. 
      For presentation, the axes are omitted and the  vertical positions
      of curve $\partial P(t)$ have been  aligned
      between snapshots.
      Simulation details are given
      in Section \ref{sec:fig}.
  }
    \label{fig:circ}
  \end{center}
\end{figure*}

The numerical method evolves a discrete level-set function $\varphi_{ij}(t)$ on a fixed $N \times  N$
uniform grid,
discretizing  $\T^2$, according to \eqref{eq:cauchyphi}. The scheme for solving (\ref{eq:cauchyphi}) is described in
\cite[Part II]{sethian1996level} and is first-order in
space with respect to the uniform grid spacing $h = 1 /N.$ The
discrete fluid velocity $u_{ij}$  is obtained from (\ref{eq:u})
using a spectral collocation method, and time integration is done with
Heun's method, which is second-order and strong-stability preserving.
Overall, the numerical solver for (\ref{eq:Bs}) is second-order in time,
first-order in space.

Using this algorithm, we compute the dynamics of the temperature patch
which is initially circular. The initial curve  $\partial
P_0$ is fixed as the embedded circle of radius one-half $ \Sone(\frac{1}{2})$. The
result of the simulation is presented in Figure \ref{fig:circ}. We observe that the curve  $\partial P(t)$
forms corner-like structures in its evolution; however, our persistence
of regularity results, Corollary \ref{thm:reg1} and Thereom
\ref{thm:reg}, hold for $\partial P(t)$ in this setting as well. Thus,
if in light of  Figure \ref{fig:circ} one asks
whether there develops a curvature singularity in finite time, we provide proof that the curvature
remains bounded for all time.

\section{Proof of Lemmas \ref{lem:pointu} and \ref{lem:key}}\label{sec:key}
The expressions with  singular integrals allow us to sufficiently control $u$ and $\omega$ via the
bounds that follow.  

\begin{proof}[Proof of Lemma \ref{lem:pointu}]
  We write explicitly,
  \begin{equation}
  \label{eq:pK}
  G(z) = \nabla^{\perp} K(z) = \frac{1}{8\pi} \begin{pmatrix} 2 \hat{z}_{1} \hat{z}_{2} \\
     1 - 2 \log |z| -  2\hat{z}^2_{1} \end{pmatrix},
  \end{equation}
  where $\hat{z}_{j} = z_{j}/|z|$, and  observe
  \begin{equation}
  \label{eq:pKabs}
  | G(x - y) | \leq C (1 + |\log |x - y | |).
  \end{equation}
  It follows immediately
  \begin{equation}
  \label{eq:pK1}
  \int\limits_{|x - y| \leq 1} |G(x - y) \theta(y) | \, dy \leq C \|\theta
  \|_{L^{\infty}}.
  \end{equation}
  Further if $|x| \geq R + 1$, then $1 \leq |x - y| \leq 2 |x|$, and we
  can bound  $|G(x - y)|$ by $ C(1 + \log (2 | x|))$. Now consider $|x| \leq R +1$ with  $|x
  -y| \geq 1$, we have instead the bound  $ C (1 + \log(2R + 1))$.
  Combining these estimates we deduce,
  \begin{equation}
  \label{eq:pK2}  
  \int\limits_{|x - y| \geq 1} |G(x - y) \theta(y) | \, dy \leq C(1 +
  \log(|x| + R + 1)) \|\theta
  \|_{L^1}.
  \end{equation}
  The result follows from (\ref{eq:pK1}) and (\ref{eq:pK2}).
\end{proof}
\begin{prop}
\label{prop:vortbdd}
  Suppose $\theta \in Y$. Let $\omega = Q * \theta$. Then
  for $2 < p \leq \infty$,
  \begin{equation*}
  \label{eq:omegaLp}
    \|\omega\|_{L^{p}} \leq C_{p} \|\theta\|_{Y}.
  \end{equation*}
\end{prop}
\begin{proof}
  We split the integral \begin{equation}
\label{eq:splitT}
- 2\pi \omega(x) =    \int\limits_{|x - y| < 1} \frac{x_1 - y_1}{|x -
    y|^2} \theta(y) \, dy + \int\limits_{|x - y |
\geq 1} \frac{x_1 - y_1}{|x - y|^2} \theta(y) \, dy ,
\end{equation}
and so deduce
\begin{equation}
\label{eq:modthet}
|\omega(x)| \leq  \frac{\sqrt{2}}{2\pi} \left(   \|\theta \|_{L^{1}}  + 4\| \theta
  \|_{L^{\infty}} \right).
\end{equation}
The case $p = \infty$ follows. 

Let $B  \subset \R^2$ be the unit ball centered at the origin. Then we have
 \begin{equation}
\omega(x) =  \omega \bm{1}_{B} (x) +  \omega \bm{1}_{\R^2 \setminus
B}(x)
\end{equation}
pointwise. From (\ref{eq:modthet}), it follows immediately that
\begin{equation}
\| \omega \bm{1}_{B} \|_{L^{p}} \leq C_{p} \|\theta\|_{Y}.
\end{equation}
and the convolution
%Observe the associative property of convolutions with
%  pointwise multiplication: $h  (f * g) = (h  f) * g = f * (h  g)$, if the
%  integrals  converge. 
$\omega = Q * \theta$
converges absolutely. Accordingly, we find \begin{equation}
\omega \bm{1}_{\R^2 \setminus B} = (Q\bm{1}_{\R^2 \setminus B} ) *
\theta,
\end{equation}
where the truncated kernel is
\begin{equation}
  Q \bm{1}_{\R^2 \setminus B}(z) = 
  \begin{cases}
    Q(z) &, \text{if } z \in \R^2 \setminus B \\
    0 &, \text{otherwise}
  \end{cases}.
\end{equation}

Let $2 < p < \infty$,  such that  $| Q \bm{1}_{\R^2 \setminus B}|^{p}$ is
integrable. We conclude
 \begin{equation}
  \| \omega \bm{1}_{\R^2 \setminus B}\|_{L^{p}} \leq 
C_{p} \|\theta\|_{L^{1}}
%  \| Q \bm{1}_{\R^2
%\setminus B}\|_{L^{p}} \| \theta \|_{L^{1}} = 
\end{equation}
from Young's inequality for convolutions (e.g. see
\cite[Appendix~A]{stein1970singular}).
\end{proof}
%Thus we deduce the bounds \begin{equation}
%\label{eq:modthbdd}
%  \| \omega \|_{L^{\infty} } \leq C \sqrt{\| \theta \|_{L^{\infty}} \|\theta
%  \|_{L^{1}}} .
%\end{equation}
%Taking $\omega$ to the power $p$ in (\ref{eq:splitT}) and integrating over $x$, we see
%that the first term generally converges, but the later term converges
%only when $p > 2$.

\begin{prop}
  Suppose $\theta \in Y$. Let $\nabla u = \nabla G * \theta$. Then
  for $2 < p \leq \infty$,
  \begin{equation*}
  \label{eq:omegaLp}
    \|\nabla u\|_{L^{p}} \leq C_{p} \|\theta\|_{Y}.
  \end{equation*}
\end{prop}
\begin{proof}
The singular kernel
$G$ away from the origin has gradient $\nabla G$, with homogeneity of
degree $-1$, such that, for integrable and bounded $\theta$, we have
the absolutely convergent integral
\begin{equation}
  \nabla u(x) =  \int_{\R^2} \nabla G (x -
  y)
\theta(y) \, dy,
\label{naUintT}
\end{equation}
where we compute explicitly as
\begin{equation}
\label{eq:sig}
\nabla G(z) =  - \frac{1}{4\pi |z|^4}
\begin{pmatrix}  z_2(z_1^2 - z_2^2) &  z_1^{3}
  + 3 z_1 z_2^2 \\  
z_1(z_2^2-z_1^2) &  z_2 (z_2^2 - z_1^2) \end{pmatrix} .
\end{equation}
Each entry of $\nabla G$ has the same cancellation property and homogeneity as $Q$, so we  conclude as in the proof of Proposition \ref{prop:vortbdd}.
\end{proof}

\begin{prop}
  \label{prop:DDom}
  Suppose $\theta \in L^{p}$ for some $p \in (1, \infty)$. Let $u = G *
  \theta$.  Then,
\begin{equation*}
\label{eq:domegaLp}
\|\nabla \nabla u \|_{L^{p}} \le C_{p} \|\theta\|_{L^{p}}.
\end{equation*}
\end{prop}

\begin{proof}
For $\nabla \nabla u =  (\partial_1 \nabla u, \partial_2 \nabla u
) $, we differentiate carefully to resolve the strongly singular kernel.
In particular,  let  $z = x - y$ such that
\begin{align}
  \partial_1 \nabla u(x) &= \int_{\R^2} \nabla G(z)
  \partial_1 \theta (x - z) \, dz \\
                        &= \lim_{\epsilon  \to 0}
                        \int_{|z| \geq \epsilon }
                        -\frac{\partial}{\partial z_1} \left( \nabla G(z)
                        \theta(x - z) \right)  + \theta(x - z)\frac{\partial}{\partial z_1}
                         \nabla G(z)  \, dz .
\end{align}
The first integral is
\begin{equation}
\label{eq:Eterm}
\begin{aligned}
  =& \lim_{\epsilon  \to 0}   \int_{|z| = \epsilon } \nabla G(z) \theta( x - z) (-n_1) \cdot d \sigma \\
  =& \lim_{\epsilon  \to 0}  \frac{1}{\epsilon } \int_{|z| = \epsilon }
                                                   z_{1}  \nabla 
                                                   G(z)
                                                  \theta( x - z) \, d \sigma \\
  =& \frac{1}{8} \begin{pmatrix} 0 & -3 \\ 1 & 0  \end{pmatrix}
  \theta(x).
\end{aligned}%
\end{equation}
We compute 
\begin{equation}
\partial_1 \nabla G(z) = \frac{1}{4\pi} \frac{\bm{\mathrm{H}}_1(z)}{|z|^2} \quad
\text{and} \quad \partial_2 \nabla G(z) = \frac{1}{4\pi}
\frac{\bm{\mathrm{H}}_2(z)}{|z|^2}
\end{equation}
where each entry of $\bm{\mathrm{H}}: \R^2 \to  \R^{2\times 2\times 2}$, explicitly
\begin{equation}
\bm{\mathrm{H}}_1(z) =\frac{1}{|z|^4} \begin{pmatrix} 2 z_1 z_2 (z_1^2 - 3
  z_2^2) & z_1^{4} + 6 z_1^2
    z_2^2 - 3z_2^{4} \\ z_1^{4} - 6 z_1^2
  z_2^2 + z_2^{4} & 2 z_1 z_2 (3 z_2^2 -
z_1^2)  \end{pmatrix},
\end{equation}
and
\begin{equation}
\bm{\mathrm{H}}_2 (z) =\frac{1}{|z|^4}
  \begin{pmatrix} -z_1^{4} + 6z_1^2 z_2^2
    -z_2^{4}  & 2z_2(3z_1 z_2^2 -
    z_1^{3})\\ 2 z_1 z_2 (z_2^2 - 3 z_1^2)
                    & z_1^{4} - 6 z_1^2
  z_2^2 + z_2^{4} \end{pmatrix},
\end{equation}
is homogeneous of degree zero, mean zero on the unit sphere, and
symmetric with respect to reflection.
It follows
\begin{equation}
\partial_1 \nabla  u(x)  = \frac{1}{8} \begin{pmatrix} 0 & -3 \\ 1 & 0
\end{pmatrix} \theta(x) + \frac{1}{4 \pi} \mathrm{pv}
\int_{\R^2}\frac{\bm{\mathrm{H}}_1(z)}{|z|^2}
\theta(x - z ) \, dz,
\end{equation}
and
\begin{equation}
\partial_2 \nabla  u(x)  = \frac{1}{8} \begin{pmatrix} 1  & 0 \\ 0 & -1
\end{pmatrix} \theta(x) + \frac{1}{4 \pi} \mathrm{pv}
\int_{\R^2}\frac{\bm{\mathrm{H}}_2(z)}{|z|^2}
\theta(x - z ) \, dz.
\end{equation}
With the appropriate $\bm{\mathrm{E}} \in \R^{2\times 2\times 2}$, we may write
compactly
\begin{equation}
\label{eq:gradUkernel}
\nabla  \nabla  u(x)  =  \theta(x) \bm{\mathrm{E}} + \frac{1}{4 \pi} \mathrm{pv}
\int_{\R^2}\frac{\bm{\mathrm{H}} (z)}{|z|^2}
\theta(x - z ) \, dz.
\end{equation}
We conclude by examining this expression as a sum of operators. The
first summand is bounded like the identity, and the second is a
Calder\'{o}n-Zygmund operator.
\end{proof}

\begin{cor}
  \label{cor:all}
  Suppose $\theta \in Y$. Let $\nabla u = \nabla G * \theta$. Then
     for $p \in (2, \infty)$
\begin{equation*}
\label{eq:GradLp}
  \|  \nabla u\|_{W^{1,p}} \leq C_{p} \|\theta\|_{Y}.
\end{equation*} 
\end{cor}

\begin{proof}[Proof of Lemma \ref{lem:key}]
To deduce the desired H\"older continuity of
$\nabla u$, we observe Corollary \ref{cor:all} and  recall  Morrey's embedding: $W^{1,p} \subset C^{\mu}$ for
$\mu = 1 -2/ p$ with  $p
>2$.   
\end{proof}

\section{Proof of Theorem \ref{thm:well}}\label{sec:well}
Let us first define what we mean by a (Yudovich) weak solution:
\begin{defi}
	Let $\theta_0 \in Y$, then  $\theta(x,t)$ is a weak solution of system
  (\ref{eq:Bs}) given $$\theta \in L^{\infty}(0,T; L^{1} \cap L^{\infty})$$ for some $T > 0$, and that for any $\phi \in C^{1}(0,T;C^{1}_c)$, the following holds:
			\begin{equation*}
			\label{eq:soln}
			\int_{\R^2} \theta(x,T) \,\phi(x,T) \, dx - \int_{\R^2} \theta_0
      \,
      \phi(x,0) \, dx = \int_0^T \int_{\R^2} \theta \,
      \left(\frac{\partial}{\partial t}  +
      u \cdot \nabla \right) \phi	 \, dx \, dt,
			\end{equation*}
      where $u(t) = G * \theta(t)$.
\end{defi}
Notice from the definition that Lemma \ref{lem:key} plays a powerful role in proving
the well-posedness of weak solutions with compactly supported initial
data. If we can maintain that $\theta(t)$ remains compactly supported in
its evolution, then the  characteristics for $\theta$ have sufficient regularity to preserve $L^p$ norms via the transport equation. With this observation, we obtain a sequence of smooth, compactly supported initial
data $\theta_0^{\epsilon}$ with regularization parameter $\epsilon$.
For such data, we obtain a family of global
smooth solutions $\theta^{\epsilon}$  with 
control of \begin{equation}
\label{eq:epcont}
\supp\theta^{\epsilon}(t) \subset B(0,R(t)) \quad\text{and}\quad
u^{\epsilon}(t) \in C^{1+\mu}(B(0,R(t)))
\end{equation}   for all time, depending only on $\|\theta_0\|_Y$. Then, the limit $\epsilon \to 0$
yields a Yudovich weak solution via the classical arguments in
\cite{yudovich1963non}. 
The limiting solution $\theta$ (possibly non-unique)
inherits these uniform controls, and these controls are strong enough to
achieve uniqueness. 
This section is dedicated to proving the claim of global well-posedness
of classical solutions
with uniform in $\epsilon $ control of \eqref{eq:epcont} 
(Proposition \ref{prop:smoothsoln}).

We begin by considering solutions to the sequence of linear equations
\begin{equation}
  \left(\frac{\partial}{\partial t} + u^{n}\cdot\nabla\right) \theta^{n+1} = 0,
\label{tneq}
\end{equation}
where $u^{n}(t) := G * \theta^{n}(t)$
%\begin{equation}
%u^n(x) = \int_{\R^2}G(x-y)\theta^n(y)dy,
%\label{unint}
%\end{equation}
%i.e.
%\begin{equation}
%u^n = \P(-\D)^{-1}\left(\theta^n e_2\right)
%\label{unt}
%\end{equation}
for positive integers $n$, defined inductively.
The initial data are fixed as $\theta^{n}(0) = \theta_0$ uniformly in $n$,
where
\begin{equation}
\label{eq:idcond}
\theta_0 \in H^{1} \cap Y, \quad \text{and} \quad \supp \theta_0 \subset B(0, R_0).
\end{equation}
For $n=0$, take $\theta^0(t) = \theta_0$.

\begin{prop}
  \label{prop:sequni}
  The sequence of functions $\theta^{n}$ described above is
  well-defined. Moreover, the following uniform estimates hold.
  \begin{enumerate}
    \item For $1 \leq p \leq \infty$,
      \begin{equation*}
      \label{eq:seqLpT}
      \|\theta^{n}(t)\|_{L^{p}} \leq \|\theta_0\|_{L^{p}}.
      \end{equation*}
    \item For $R(0) = R_0$, we have 
      \begin{equation*}
      \label{eq:seqSuppT}
      \supp \theta^{n}(t) \subset B(0,R(t))
      \end{equation*}
      with $R(t)$ obeying the differential inequality
       \begin{equation*}
      \label{eq:dR}
      \frac{d R(t)}{d t} \leq C (1 + \log (2R(t) + 1)) \|\theta_0\|_{Y}.
      \end{equation*}
    \item
      For $\mu \in (0,1)$,
      \begin{equation*}
      \label{eq:estthree}
      \|\nabla u^{n}(t)\|_{C^{\mu}} \leq C_{\mu} \|\theta_0\|_{Y} 
      \end{equation*}
      where $u^{n}(t) = G * \theta^{n}(t)$.
  \end{enumerate}
\end{prop}
\begin{proof}
  For $n = 0$, all the estimates
  are immediate as  $\theta^{0} = \theta_0$. Now letting $n \geq 0$,
  suppose  $\theta^{n}$ satisfies the estimates of the theorem.

  For the first estimate, the regularity of  $u^{n}$ in the linear
  transport equation (\ref{tneq}) produces a unique global
  solution $\theta^{n+1}$, and  $\|
  \theta^{n+1}\|_{L^{\infty}}$ is preserved along characteristics. Further, we may multiply
  (\ref{tneq}) by $\theta |\theta|^{p-2}$, for $p \geq 2$, and integrate
  by parts  to discover 
  \begin{equation}
  \label{eq:seqLpTnn}
    \frac{d}{dt} \|\theta^{n+ 1}\|_{L^{p}} \leq 0,
  \end{equation}
  since $u^{n}$ is divergence-free. We
  deduce $\theta^{n+1}$
  also satisfies Estimate \ref{eq:seqLpT}, as desired.

  For the second estimate, we observe $u^{n} = G * \theta^{n}$ implies
  \begin{equation}
  \label{eq:sequdR}
|u^{n}(x,t)| \le C(1 + \log(|x|+ R(t) + 1))\|\theta^{n}(t)\|_Y.
  \end{equation}
Indeed $\theta^{n} \in Y$, and the previous arguments produce
 \begin{equation}
\label{eq:seqLYT}
\|\theta^{n+1}(t)\|_{Y} \leq \|\theta_0\|_{Y}.
\end{equation}
Moreover, $\theta^{n+1}$ is transported by $u^{n}$, which we showed is
bounded.  Let $X(a,t)$ be the flow map generated by $u^{n}$ with labels
$a \in \R^2$, such that
\begin{equation}
\label{eq:flowmap}
\frac{d X(a,t)}{dt} = u^{n}(X(a,t),t)
\end{equation}
where $X(a,0) = a$. Observe 
\begin{equation}
\label{eq:RX}  
\supp \theta^{n+1}(t) \subset \{ x  \mid  x = X(a,s),\, s \in \lbrack 0,
  t \rbrack, \, a \in B(0, R_0)
\}   ,
\end{equation}
whereby the inequality (\ref{eq:flowmap}) gives
\begin{equation}
\label{eq:seqdR}
\begin{aligned}
  \frac{dR(t)}{dt} &\leq \sup_{a \in B(0,R_0)} |u(X(a,t),t)|  \\
  &\leq  C  (1 + \log(2 R(t) + 1)) \| \theta^{n}
(t)\|_{Y},
\end{aligned}
\end{equation}
such that Estimate \ref{eq:seqLpT} implies $\theta^{n+1}$ satisfies Estimate \ref{eq:seqSuppT} in the
Proposition.

For Estimate \ref{eq:estthree}, we may simply apply Lemma \ref{lem:key} and the previous
conclusions to deduce
\begin{equation}
\|\nabla u^{n+1}(t)\|_{C^{\mu}} \leq C_{\mu} \|\theta_0\|_{Y}
\end{equation}
for $\mu \in (0,1)$ where $u^{n+1}(t) = G * \theta^{n+1}(t)$.
The results follow from induction.
\end{proof}

\begin{prop}
  \label{prop:unigradt}
  The sequence of functions $\theta^{n}$ satisfies the  uniform
  estimate 
\begin{equation*}
\|\nabla\theta^{n}(t)\|_{L^{p}} \leq
\|\nabla\theta_0\|_{L^{p}}\exp(C_{L}t)
\end{equation*}
for $1\leq p\leq \infty$.
\end{prop}
\begin{proof}
  
Differentiating the transport equation (\ref{tneq}) yields
\begin{equation}
\label{eq:gradtransport}
\left( \frac{\partial}{\partial t} + u^{n-1} \cdot \nabla  \right)  \nabla^{\perp} \theta^{n} 
%= \partial_t \nabla^{\perp} \theta + (u \cdot \nabla) \nabla^{\perp} \theta 
= (\nabla u^{n-1}) \nabla^{\perp} \theta^{n}.
\end{equation}
 We then multiply the above equation by
$\nabla^{\perp} \theta^{n}|\nabla^{\perp}\theta^{n}|^{p -2} $ for $p
\geq 2$, and integrate by parts to find
\begin{equation}
\label{eq:LpgradTp}
	\begin{aligned}
	 \frac{d}{dt} \| \nabla^{\perp} \theta^{n}\|^{p}_{L^{p}} 
%   &=
%   \int_{\R^2} (u \cdot \nabla ) |\nabla \theta^{n} |^{p} \, dx + p
%   \int_{\R^2} (\nabla^{\perp} \theta^{n} \cdot  \nabla ) u^{n-1} \cdot
%   \nabla^{\perp}\theta^{n} | \nabla^{\perp} \theta^{n} |^{p -2} \, dx \\
																																		%&= -(p-1) \int_{\R^2} u \cdot ( \nabla^{\perp} \theta \cdot \nabla ) \nabla^{\perp} \theta | \nabla^{\perp}	 \theta |^{p - 2} \, dx \\
																																		%&\leq (p - 1) \int_{\R^2} |u| | \nabla^{\perp} \theta | | D^2 \theta | | \nabla^{\perp} \theta |^{p -2} \, dx.
					&\leq p \int_{\R^2}|\nabla u^{n-1}||\nabla \theta^{n}
          |^{p}\,dx .
	\end{aligned}
\end{equation}
Therefore \begin{equation}
\label{eq:LpgradT}
	\frac{d}{dt}\|\nabla \theta^{n}\|_{L^{p}} \leq \|\nabla
  u^{n-1}\|_{L^{\infty}} \|\nabla \theta^{n}\|_{L^{p}},
\end{equation}
for $1 \le p \le \infty$, where the cases $p = 1$ and  $p = \infty$ are
direct from (\ref{eq:gradtransport}).
From here, Gr\"{o}nwall's Lemma yields
\begin{equation}
\label{eq:gronwallGrad}
\|\nabla \theta^{n}(t)\|_{L^{p}} \leq \|\nabla \theta_0\|_{L^{p}}
\exp\left( \int_0^{t} \|\nabla u^{n-1}(s)\|_{L^{\infty}} \, ds  \right).
\end{equation}
By the third estimate in Proposition \ref{prop:sequni}, it follows $\nabla u^{n-1}$ is
bounded absolutely by a  constant  $C_{L}$ which
depends only on $ \|\theta_0\|_{Y}$.
\end{proof}
\begin{prop}
\label{prop:L2}
  For any $t$, the sequence $\theta^{n}(t)$ converges strongly in $L^{2}$  to some
  function  $\theta(t)$. Moreover, $\theta$ obeys Estimates
  \ref{eq:LpT}, \ref{eq:suppT} and \ref{eq:cUT} in
  Theorem \ref{thm:well}.
\end{prop}
\begin{proof}
  
We denote 
\begin{equation}
\vartheta^ {n+1} = \theta^{n+1}-\theta^n
\label{deln}
\end{equation}
and conclude by showing the sequence $\vartheta^ {n}(t)$ is summable
in $L^2 $ for any $t$. Accordingly, we define $v^{n} = G * \vartheta^{n}$ such that equation obeyed by $\vartheta^ {n+1}$ reads
\begin{equation}
\partial_t \vartheta^ {n+1} + \langle u\rangle^{n}\cdot\nabla\vartheta^ {n+1} + v^{n}\cdot\nabla\langle\theta\rangle^{n+1} = 0
\label{delneq}
\end{equation}
where
\begin{equation}
\langle u\rangle^{n} = \frac{1}{2}\left( u^n + u ^{n-1}\right)
\label{Un}
\end{equation}
and 
\begin{equation}
\langle\theta\rangle^{n+1} = \frac{1}{2}\left (\theta^{n+1} + \theta^n\right).
\label{Tn}
\end{equation}
Taking the $L^2$ inner product of (\ref{delneq}) with $\vartheta^ {n+1}$
gives us
\begin{equation}
\frac{d}{dt}\|\vartheta^ {n+1}(t)\|_{L^2}^2 = -2\int_{\R^2} v^{n}\cdot\nabla\langle\theta\rangle^{n+1}\vartheta^ {n+1} dx 
\label{l2}
\end{equation}
after integration by parts, noting $\langle u\rangle^{n}$ is divergence-free.

%From (\ref{l2}),  we have
%\begin{equation}
%\label{eq:dtL2}
%\frac{d}{dt}\|\vartheta^ {n+1}\|_{L^{2}} \leq C \|\theta^{n+1}
%\|_{W^{1,\infty}} \|v^{n}\|_{L^2},
%\end{equation}
%where we used Holder's inequality. Integrating in time and observing
%(\ref{grdunb}) with (\ref{eq:seqLpT}),  we discover 
%
%
%\begin{equation}
%\|\vartheta^ {n+1}(t)\|_{L^2} \le C_T\int_0^t \|v^n(s)\|_{L^2}ds
%\label{delnb}
%\end{equation}
%holds for any $t\le T$, where $T$ is arbitrary. Here, the constant $C_T$
%depends on $\|\theta_0\|_{W^{1,\infty}}$ but is independent of $n$.

We estimate the convolution $G * \vartheta^{n}$ using the fact that the supports are in the ball of radius $R(t)$. 
We obtain
\begin{equation}
|v^n(x,t)| \le \|\vartheta^ n(t)\|_{L^2} \left(\int_{|y| \leq   R(t)} |
G(x - y)
|^2 \, d y \right)^{1/2}
\label{vnb}
\end{equation}
from the Cauchy-Schwarz inequality. Using Lemma \ref{lem:pointu} and evaluating the integral, we discover
\begin{equation}
|v^n(x,t)| \le C(1+\log (2R(t) + 1)) R(t) \|\vartheta^ n(t)\|_{L^2}.
\label{vnb}
\end{equation}
Now the righthand side of (\ref{l2}) has the bound,
\begin{equation}
\begin{array}{l}
\left |-2\int_{\R^2} v^{n}\cdot\nabla\langle\theta\rangle^{n+1}\vartheta^ {n+1} dx\right| \\
\le C(1+\log(2R(t)+1)) R(t) \|\nabla\langle\theta\rangle^{n+1}(t)\|_{L^2}\|\vartheta^ n(t)\|_{L^2}\|\vartheta^ {n+1}(t)\|_{L^2}.
\end{array}
\label{rhs}
\end{equation}
The estimate then becomes,
\begin{equation}
\label{eq:nathetana}
\frac{d}{dt} \Theta^{n+1}(t) \leq  A_{T} \|\nabla \langle\theta\rangle^{n+1}\|_{L^2}
\Theta^{n}(t),
\end{equation}
where $\Theta^{n} = \| \vartheta^{n} \|_{L^{2}}$, and $A_{T} =
C(1+\log(2R(T)+1)) R(T)$.
Using Proposition \ref{prop:unigradt} for $p=2$, we write the integral expression for
all $t \leq T$
\begin{equation}
\label{eq:nathetanint}
\Theta^{n+1}(t) \leq  A_{T} \|\nabla \theta_0\|_{L^{2}} \exp(C_{L}T)
\int_0^{t} \Theta^{n}(s) \, ds,
\end{equation}
noting that $\Theta^{n}(0) = 0$ for all $n$. 

Denoting $C_{T} =
A_{T}
\| \nabla \theta_0\|_{L^{2}} \exp(C_{L}T)$, we compute $n = 1$:
\begin{equation}
\label{eq:firstn}
  \begin{aligned}
    \Theta^{2}(t) &\leq C_{T} \int_0^{t} \Theta^{1}(s) \, d s \\
                  &\leq C_{T} \left(  \sup_{t \leq T}\Theta^{1}(t)
                  \right) t \\
                  &\leq C_{T} C_{1} t
  \end{aligned}
\end{equation}
where $C_{1}$ depends only on  $\|\theta_0\|_{L^2}$.
Suppose now for some $n$,
\begin{equation}
\label{eq:middlen}
\Theta^{n}(t) \leq C_{1} \frac{(C_{T} t)^{n}}{n!},
\end{equation}
then we have
\begin{equation}
\label{eq:nextn}
  \begin{aligned}
    \Theta^{n+1}(t) &\leq C_{1} C_{T}^{n+1} \int_0^{t}
    \frac{t^{n}}{n!} \\
                    &\leq C_{1} \frac{(C_{T} t)^{n+1}}{(n+1)!}.
  \end{aligned}
\end{equation}

It follows by induction that estimate (\ref{eq:middlen}) holds for all  $n
 \geq 1$. Indeed,
 \begin{equation}
\label{eq:sumn}
\sum_{n=1}^{\infty} \Theta^{n+1}(t) \leq C_{1} \exp(C_{T} t) 
\end{equation}
such that $\Theta^{n}(t)$ is summable for all $t \leq T$, where  $T$ is
arbitrary. Therefore, the sequence $\theta^{n}(t)$ converges strongly in $L^{2}$
 to a function $\theta(t)$ for any $t$, and the uniform
bounds in Proposition \ref{prop:sequni} are inherited by
the limit.
\end{proof}
In view of Lemma \ref{lem:pointu}, the linear operator given by kernel $G$ is generally not bounded in
 $L^{2}$. Towards a convergence in more regular spaces, an immediate
corollary of Proposition \ref{prop:sequni} is that $\|\nabla \nabla u^{n} \|_{L^2} \leq C
\|\theta_0\|_{Y}$, uniformly
in $n$. We further establish the higher regularity estimate:
\begin{lemma}
  \label{lem:seqapriori}
 Suppose $\theta \in Y \cap H^{m}$ for some positive integer $m$. Let $u =
 G * \theta$. Then,
 \begin{equation*}
    \|D^{m+2} u\|_{L^2} \leq  \|\theta\|_{H^{m}}.
 \end{equation*}
\end{lemma}

\begin{proof}  
 We first observe that $u = G * \theta$ is a weak solution to
\begin{equation}
\label{eq:unpde}
\begin{aligned}
  -\Delta u   + \nabla \Pi &=  \theta e_2, \\
\end{aligned}
\end{equation}
where $\Pi$ is given by $\Delta\Pi = \partial_2 \theta$. Then, we apply $D^{m}$ to the above equation and take the $L^{2}$ inner product
with $D^{m} (-\Delta u)$ to arrive at
\begin{equation}
\begin{aligned}
\int_{\R^2} (D^{m} ( \Delta u)) \cdot 
D^{m} (\Delta u)  \,  d x=\ 
  \int_{\R^2} (D^{m} ( \theta e_2)) \cdot 
  D^{m} (\theta e_2)  \,  d x,
\end{aligned}
\end{equation}
whereby integration by parts twice gives
\begin{equation}
   \|D^{m+2} u \|_{L^2} = \|D^{m}\theta \|_{L^2}.
\end{equation}
\end{proof}
For higher regularity initial data, we show the limit $\theta$ indeed
solves our system. 
The  precompactness for $\theta^{n}$ we require follows from $H^{m}$-energy estimates, where
the $H^{0}$ estimate is immediate from Estimate \ref{eq:seqLpT} in
Proposition \ref{prop:sequni}.  

\begin{prop}
  \label{prop:highnorm}
  Let $\theta_0 \in Y \cap H^{m}$ for some integer $m > 2$.
 Then, the sequence $\theta^{n}$ satisfies
  the uniform estimate
\begin{equation*}
  \|\theta^{n}(t)\|_{H^{m}} \leq C_{m} \|\theta_0\|_{H^{m}} \exp(C_m
  \exp(C_L |t|)),
\end{equation*}
where $C_{m}$ depends only on $m$ and  $\theta_0$.
\end{prop}

\begin{proof}
Further, we apply the operator
$D^{\alpha}$ with multi-index $\alpha$ to (\ref{tneq}), multiply by $D^{\alpha} \theta^{n+1}$,
integrate over $\R^2$, and sum over $|\alpha | \leq m$ to discover  
\begin{equation}
\label{eq:dynHs}
\begin{aligned}
  \frac{1}{2} \frac{d}{dt} \|\theta^{n+1} \|^2_{H^{m}} &= -
\sum_{|\alpha|\leq m} \int_{\R^2} (D^{\alpha} ( u^{n} \cdot \nabla
\theta^{n+1})) D^{\alpha} \theta^{n+1}  \,  d x \\
                                                       &= -
                                                       \sum_{|\alpha|
                                                       \leq m }
                                                       \int_{\R^2}
                                                       (D^{\alpha}(u^{n}\cdot
                                                       \nabla
                                                       \theta^{n+1}) -
                                                       u^{n}\cdot \nabla
                                                       (D^{\alpha}
                                                       \theta^{n+1}))
                                                       D^{\alpha}
                                                       \theta^{n+1} \,
                                                       dx, 
\end{aligned}
\end{equation}
since $u^{n}$ is divergence-free. Using calculus inequalities (e.g. see 
\cite[Lemma 3.4]{majda2002vorticity}) gives us that
\begin{equation}
\label{eq:Hsineq}
  \begin{aligned}
    \frac{1}{2} \frac{d}{dt} \|\theta^{n+1}\|_{H^{m}} &\leq
    \sum_{|\alpha| \leq m} \|
    D^{\alpha}(u^{n}\cdot \nabla \theta^{n+1}) - u^{n} \cdot \nabla
    (D^{\alpha} \theta^{n+1})\|_{L^2} \\
          &\leq C_{m} \Big( \|\nabla u^{n}\|_{L^{\infty}} \|D^{m}\theta^{n+1}\|_{L^2} + \|D^{m} u^{n} \|_{L^2}\|\nabla \theta^{n+1}\|_{L^{\infty}}\Big).
  \end{aligned}
\end{equation}
With  Proposition \ref{prop:unigradt}, Lemma \ref{lem:seqapriori}, and the
embedding $H^{m} \subset L^{\infty}$ for $m > 1$, we find 
\begin{equation}
\label{eq:dynmollHs}
\frac{d}{dt} \|\theta^{n+1} \|_{H^{m}} \leq  C_{m} (\| D^{m} \theta^{n+1
} \|_{L^2} + \exp(C_{L}|t|)\|\theta^{n}\|_{H^{m-1}} ),
\end{equation}
where $C_m$ depends only on $m$ and $\theta_0 \in  Y \cap H^{m}$.
The result follows from iteration.
\end{proof}
\begin{prop}
  \label{prop:smoothsoln}
  Let $\theta_0 \in Y \cap H^{m}$ for some integer $m > 2$,
  and suppose that $\supp \theta_0 \subset
  B(0,R_0).$ Then for arbitrary $T$, there
  exists a unique classical solution
  $$\theta \in C^{1}(0, T;
    C^{1}_0),$$
  to system (\ref{eq:Bs}). Moreover, $\theta$ satisfies Estimates
  \ref{eq:LpT}, \ref{eq:suppT} and \ref{eq:cUT} in
  Theorem \ref{thm:well}.
\end{prop}
\begin{proof}
Because of Proposition \ref{prop:highnorm}, the sequence $\theta^n(t)$ has a subsequence $\theta^{n_j}(t)$ which converges strongly in $H^m$ to the limit $\theta(t)$ in Proposition \ref{prop:L2}, for $0\leq t \leq T$. As $m >2$, the continuous embedding
\begin{equation}
H^{m+k} \subset C^{k}_0 \equiv \{\text{space of }C^k\text{ functions vanishing at infinity} \}
\end{equation}
implies $\theta^{n_j}(t)$ converges strongly in $C^1_0$, and $\theta(t) \in C^1_0$ satisfies Estimates
  \ref{eq:LpT}, \ref{eq:suppT} and \ref{eq:cUT}.

Towards verifying $\theta$ is a solution, we define $u = G * \theta$ and observe that
$\theta^n_j\in Y$ and $\theta \in Y$ are both supported in $B(0,R)$ where $R = R(t)$ is given by Estimate \ref{eq:suppT}. At each instance of time,  Lemma \ref{lem:pointu} implies
\begin{equation}
|u^{n_j}(x) -  u(x)|  \leq C (1 + \log (|x| + R + 1)) \|\theta^{n_j} - \theta \|_{Y}.
\end{equation}
Then, we may conclude $u^{n_j} \to u$ pointwise via the continuous embedding $L^1_{\mathrm{loc}} \subset H^1$. 
Estimate \ref{eq:estthree} and Azelà-Azcoli imply the convergence is uniform in $B(0,R)$.  

Similarly, we have also $\theta^{n_k-1}$ converging strongly to $\theta$ in $C^1_0$ for each fixed time, where $n_k - 1$ indexes a subsequence of $\theta^{n_j-1}$. It then follows that $u^{n_k-1}(t)  \cdot \nabla \theta^{n_k}(t)$ converges uniformly. Pointwise in space and time, we have
\begin{equation}
\label{eq:tneq2}
  \partial_{t} \theta^{n_k} = -(u^{n_k-1}) \cdot \nabla \theta^{n_k},
\end{equation}
and thus $\partial_{t} \theta^{n_k}$ converges strongly to $-u\cdot \nabla \theta$
in $C(0,T; C_0)$. The distributional limit of $\partial_{t}
\theta^{n_k}$ is accordingly $\partial_{t} \theta$; therefore $\theta$ is a classical
solution to system (\ref{eq:Bs}).
\end{proof}
\begin{proof}[Proof of Theorem \ref{thm:well}]
  Existence with the estimates follows from Proposition \ref{prop:smoothsoln} using the arguments in the classical paper of Yudovich \cite{yudovich1963non},
  basically unmodified. Here, we prove uniqueness using the estimates from
  Lemma \ref{lem:key}. Suppose with initial data $\theta_0 \in Y$ of
compact support we have two
weak solutions, $\theta_1$ and $\theta_2$. It follows that the difference  $\vartheta = \theta_1 -
\theta_2$  obeys the equations
\begin{equation}
\label{eq:avg}
\left( \frac{\partial}{\partial t} + \langle u \rangle \cdot \nabla  \right) \vartheta =
- v \cdot \nabla \langle \theta \rangle
\end{equation}
in the sense of distributions, where $v(t) = G * \vartheta(t)$ and
\begin{equation}
\begin{aligned}
\langle u  \rangle &= \frac{1}{2}(G * \theta_1 + G * \theta_2),  \\
\langle \theta \rangle &= \frac{1}{2}(\theta_{1} + \theta_{2}).
\end{aligned}
\end{equation}
At each instance of time, we have $\vartheta(t) \in Y$ and clearly 
 \begin{equation}
\supp \vartheta (t) \subset \supp \theta_1(t) \cup \supp \theta_2(t),
\end{equation}
since $\theta_1$ and $\theta_2$ obey the estimates. Through  Lemma \ref{lem:key}, the
 velocities $\langle u \rangle$ and $v$ have gradients
bounded uniformly in time with constant depending only on
$\|\theta_1\|_{Y} + \|\theta_2\|_{Y}$. We multiply equation
(\ref{eq:avg}) with $\vartheta$ and integrate by parts to discover
\begin{equation}
\frac{d}{dt}\|\vartheta \|_{L^2}^2 = 0,
\end{equation}
since $\langle u \rangle$ and $v$ are divergence-free and regular
enough. Recalling that $\| \vartheta(0)\|_{L^2} = 0$, we have proved uniqueness.
\end{proof}

\section{Proof of Theorem \ref{thm:reg}}\label{sec:reg}
We consider  the dynamics of the patch solution 
$\theta(t) = \bm{1}_{P(t)}$ given by Corollary \ref{cor:patch} and
address here the question
of regularity for the evolving boundary $\partial P(t).$ 

First, we note
that $\theta$ must necessarily satisfy the estimates in Theorem
\ref{thm:well}. 
The particle trajectories are volume-preserving such that the area of $P(t)$
is constant in time, and so we define the length scale
 \begin{equation}
   L = \sqrt{ \text{area}(P_0)} = \sqrt{ \text{area}(P(t))}.
\end{equation}
Then, $u = G * \bm{1}_{P}$ has the uniform bound in time:
\begin{equation}
\label{eq:controlL}
\|\nabla u (t) \|_{C^{\mu}} \leq C_{\mu} C_{L}
\end{equation}
where $C_{L} = 1 + L^2$ depends only on the initial data.

The regularity result for the distribution $\bm{1}_{P}$ follows from
analysis of  gradients of the defining level-set function $\varphi(t)$ satisfying
(\ref{eq:patchphi}), which is the unique global solution to the Cauchy problem 
\begin{equation}
\label{eq:cauchyphiic}
\left\{\begin{aligned}
\left(\frac{\partial}{\partial t} + u \cdot \nabla \right) \varphi &= 0, \\
  \varphi (x,0) &= \varphi_0(x),
\end{aligned}\right.
\end{equation}
where $\varphi_0$ satisfies (\ref{eq:patchphi0}).
Above, the fluid velocity $u = G * \bm{1}_{P}$ has the convenient expression 
using $\varphi$,
\begin{equation}
\label{eq:patchu}
u(x,t) = \int_{\R^2} G (x - y) H(\varphi(y,t)) \, dy,
\end{equation}
where $H$ is the Heaviside function. By construction, $P(t)$ is
precisely the zero level-set of $\varphi(t)$ such that
\begin{equation}
\bm{1}_{P(t)} = H \circ \varphi(t),
\end{equation}
and so our regularity
results for the boundary rely on an analysis of $\W = \nabla^\perp \varphi$, which is tangent to $\partial P$. Recalling the evolution equation (\ref{eq:weq}) for $\W(t)$, we must estimate the product $\nabla u \,\W$.

We provide an additional observation in estimating $|\cdot|_\mu$: the singular kernel
$G$ away from the origin has gradient $\nabla G$, with homogeneity of
degree $-1$, such that for each instance of time 
\begin{equation}
  \nabla u(x) =  \mathrm{pv} \int_{P} \nabla G (x -
  y) \, dy,
\label{naUintP}
\end{equation}
and so we write $\nabla u = \nabla G * \bm{1}_{P}$. 
Since the degree in modulus is less than the dimension, integration by parts yields
\begin{equation}
\label{eq:} 
      \mathrm{pv} \int_{P} \nabla \lbrack G(x -
      y) \rbrack \cdot \W(y) \, dy  = 0
\end{equation}
for divergence-free $\W$ tangent to the boundary.  Thus, we have commutative
structure for the kernel $\nabla G(z)$,
\begin{equation}
\nabla u(x) \, \W = (\nabla G \, \W ) * \bm{1}_{P},
\end{equation}
and following estimate:
\begin{lemma}
  \label{lem:com1}
  Suppose $u = G * \bm{1}_{P}$ and that $\W \in C^{\mu}(\R^2,\R^2)$ is divergence free
  and tangent to $\partial P$. Then, there exists independent constant $C_2$ such
  that
  \begin{equation*}
     \lvert \nabla u \W \rvert_{\mu} \leq C_2 \|\nabla  u \|_{L^{\infty}} \lvert \W \rvert_{\mu}.
  \end{equation*}
\end{lemma}

Using the particle
trajectories of $u$, we derive from (\ref{eq:weq}) pointwise estimates on equation $\W$ (see
 \cite[Proposition 3]{Bertozzi1993GlobalRF}) wherein Gr\"{o}nwall's inequality yields
bounds for quantities \eqref{eq:quant1} and \eqref{eq:quant2} which depend exponentially on
\begin{equation}
\label{eq:Duquant}
\int_{0}^{t} \|\nabla u (s)\|_{L^{\infty}} \,  d s \leq C_{L} \lvert t\rvert.
\end{equation}
The particle trajectories allow us to deduce the
regularity
result: 
\begin{prop}
  \label{prop:c1}
  Suppose $P_0$ bounded and $\varphi_0 \in C^{1 + \mu} (\R^2) $
  in (\ref{eq:cauchyphiic}) such that $\lvert\W_0\rvert_{\inf} > 0$. Then, 
  %there exists a
  %constant $C_{L}$ depending only on  $L$ such that 
  the unique global
  solution $\varphi$  has $\W = \nabla^{\perp}
  \varphi$ which satisfies:
\begin{align*}
  &\lvert\W(t)\rvert_{\mu} \leq \lvert\W_0\rvert_{\mu} \exp
  ((C_2 + \mu)C_{L}\lvert t\rvert ), \\
  & \|\W(t)\|_{L^{\infty}} \leq \|\W_0\|_{L^{\infty}} \exp
  (C_{L} \lvert t\rvert),  \\
  &\lvert\W(t)\rvert_{\inf} \geq \lvert\W_0\rvert_{\inf} \exp
  (-C_{L} \lvert t\rvert).
\end{align*}
\end{prop}
\begin{proof}[Proof of Theorem \ref{thm:reg} for $\partial P_0 \in C^{1+\mu}$]
 Indeed, Proposition \ref{prop:c1} guarantees the desired $C^{1+\mu}$
 parametrization $\z$ in  (\ref{eq:parah})
 exists at each instance of time, if $\varphi_0$ in
 (\ref{eq:cauchyphi}) is chosen
 such that $\lvert\W_0\rvert_{\inf} = |\nabla^{\perp}\varphi_0|_{\inf}$ is non-vanishing.
\end{proof}
\begin{rem}
Lemma \ref{lem:com1} is used only in the estimate of $ | \W|_{\mu}$ in
Proposition \ref{prop:c1} and  thus unnecessary to deduce
 $\partial P \in C^{1+ \mu}$. The infimum and supremum bounds are direct
  from (\ref{eq:weq}) and $\nabla u \in C^{\mu}$ depends only on
  $\theta_0$ so that we may alternatively conclude with an $\exp(C_{L} \exp{(C_{L}
  \lvert t\rvert)})$ 
  bound on $ | \W|_{\mu}$.
\end{rem}
For $C^{2 + \mu}$-regularity, we must examine the dynamics of $\nabla
\W$. The presence of $\nabla
\nabla u \cdot \W$ in equation  (\ref{eq:dynDW})  augments the
Gr\"{o}nwall-type
exponential bounds from particle trajectories. Inspecting expression (\ref{eq:gradUkernel}) for $\nabla \nabla u$ in this context,
\begin{equation}
\label{eq:DDuintP}
\nabla  \nabla  u(x)  =  \bm{1}_{P}(x) \bm{\mathrm{E}} + \frac{1}{4 \pi} \mathrm{pv}
\int_{P}\frac{\bm{\mathrm{H}} (x-y)}{|x-y|^2} \, dy,  
\end{equation}
we see that 
$| \nabla \nabla u |_{\mu} $ is difficult to
estimate directly. From here, we recognize the
expression for $\nabla \nabla u(x) $ has similar geometric properties
properties to the strain tensor for vortex patches \cite{Bertozzi1993GlobalRF}. We adapt
their arguments here.

\begin{prop}
  Suppose $u = G * \bm{1}_{P}$ and that $
  \W \in C^{\mu}(\R^2,\R^2)$ is divergence free and tangent to $\partial
  P$. Then 
  \begin{equation*}
    \nabla \nabla u(x) \cdot \W = \frac{1}{4\pi} \mathrm{pv} \int_{P}
  \frac{\bm{\mathrm{H}}(x - y)}{|x - y|^{2}} \cdot  (\W(x) - \W(y)) \, d y.
  \end{equation*}
\end{prop}
\begin{proof}
   Since $\W$ is
  divergence free and tangent to  $\partial P$, then
  \begin{equation}
      \begin{aligned}
      \mathrm{pv} \int_{P} &\nabla \lbrack \nabla G(x -
      y) \rbrack \cdot \W(y) \, dy \\
                    &= - \lim_{\delta \to 0} 
                    \int\limits_{|x - y| = \delta, \,y \in P} \bigg(\W(y) \cdot
                    \left(\frac{x-y}{\delta}\right)\bigg) \nabla G(x -
                    y) \, dy \\
                    &= - \bm{1}_{P}(x) \bm{\mathrm{E}} \cdot \W(x).
    \end{aligned}
  \end{equation}
  The last equality follows from (\ref{eq:Eterm}) in the proof of
  Proposition  \ref{prop:DDom}.
\end{proof}
\begin{cor}
  \label{cor:com2}
  Suppose $u = G * \bm{1}_{P}$ and that $
  \W \in C^{\mu}(\R^2,\R^2)$ is divergence free and tangent to $\partial
  P$. Then, there exists independent constant $C_3$ such that 
  \begin{equation*}
    |\nabla \nabla u \cdot  \W |_{\mu} \leq C_3 \| \nabla \nabla u\|_{L^{\infty}} |\W |_{\mu}.
  \end{equation*}
\end{cor}
The estimation of $\| \nabla \nabla u \|_{L^{\infty}}$ is consequence of
of the fact that small neighborhoods containing $\partial P$ look like
half-circles,  and that the kernel $\bm{\mathrm{H}}(z) / |z|^2$ is reflection
symmetric. To
illustrate, consider
the set of points $x_0$ with distance
\begin{equation}
d(x_0) = \inf\limits_{x \in \partial P} \{|x - x_0|\} 
\end{equation}
less than a cutoff $0 < \delta \leq \infty$. This cutoff is explicit
\begin{equation}
  \delta^{\mu} = \frac{| \nabla \varphi|_{\inf}}{| \nabla \varphi
  |_{\mu}},
\end{equation}
given $\varphi$ which satisfies (\ref{eq:patchphi}).
Such a choice ensures that the boundary $ \partial P$ can
be straightened near the points $x_0$ where
\begin{equation}
\label{eq:deldef}
d(x_0) < \delta.
\end{equation}

Indeed, for such points $x_0$ the semicircle
\begin{equation}
\Sigma(x_0) = \{z  \mid  |z| =1, \nabla \varphi(\tilde{x}) \cdot z \geq
0\} ,
\end{equation}
where above $\tilde{x} \in \partial P$ is any such that
$|\tilde{x}- x_0| = d(x_0)$ holds, 
is in fact well-approximated by the set of directions
\begin{equation}
S_{\rho}(x_0) = \{z  \mid |z|=1, x_0 + \rho z \in P\} 
\end{equation}
for every $\rho \geq d(x_0)$. More quantitatively,
%The proof of the Corollary may be found in the Appendix of
%\cite{Bertozzi1993GlobalRF}. 
\begin{lemma}[Geometric Lemma]
  Denote the symmetric difference as 
\begin{equation*}
\label{eq:symdif}
R_{\rho}(x_0) = (S_{\rho}(x_0) \setminus \Sigma(x_0)) \cup (\Sigma(x_0)
\setminus S_{\rho}(x_0))
\end{equation*}
and the Lebesgue measure on the unit circle as $H^{1}$.
  Then,
  \begin{equation*}
    H^{1}(R_{\rho}(x_0)) \leq 2 \pi \left( (1 + 2^{\mu})
  \frac{d(x_0)}{\rho} + 2^{\mu} \left( \frac{\rho}{\delta} \right)^{\mu}  \right) 
  \end{equation*}
  for all $\rho \geq d(x_0), \mu > 0$ and $x_0$ such that $d(x_0) <
  \delta = \left( \dfrac{| \nabla \varphi |_{\inf}}{|\nabla \varphi
  |_{\mu}} \right)^{1/ \mu} .$ 
\end{lemma}
\begin{prop}
  \label{prop:DDu}
  Suppose $u = G * \bm{1}_{P}$ and that $\varphi$ satisfies (\ref{eq:patchphi}) for
  $P$. Then, there exists constant $C_4$ depending only on  $ \mu, L, \lvert\W_0\rvert_{\mu}$ and $ |\W_0
  |_{\inf}$ such that
  \begin{equation*}
    \|\nabla \nabla u(t)\|_{L^{\infty}} \leq C_4 ( 1 + |t|).
  \end{equation*}
\end{prop}
\begin{proof}
 We need only estimate the singular integral, which we split into
 $I_1$
 and $I_2$. For $\delta$ given by (\ref{eq:deldef}), the latter has the
 bound
 \begin{equation}
 |  I_2(x_0)|= \left| \frac{1}{4 \pi} \int_{P \cap \{| x_0 - y | \geq
 \delta\} } \frac{\bm{\mathrm{H}}(x_0 - y)}{|x_0 - y|^2} \, dy \right|
 \leq 1 + \log\left( \frac{\delta}{L} \right).
 \end{equation}

 The remaining term,
\begin{equation}
| I_1(x_0)| = \frac{1}{4 \pi} \int_{P \cap \{| x_0 - y | <
 \delta\} } \frac{\bm{\mathrm{H}}(x_0 - y)}{|x_0 - y|^2} \, dy
\end{equation}
vanishes for $d(x_0) > \delta$. We thus assume $d(x_0) < \delta$, and pass to
polar coordinates centered at $x_0$ to find
\begin{equation}
| I_{1}(x_0) | \leq \frac{1}{4\pi} \int_{d(x_0)}^{\delta}
\frac{d \rho}{ \rho} H^{1}(R_{\rho}(x_0)),
\end{equation}
using the fact $\int_{\Sigma(x_0)} \bm{\mathrm{H}}(z) d H^{1}(z)$
vanishes by
reflection symmetry. Now applying the Geometric Lemma,
we integrate to discover
\begin{equation}
| I_{1} | \leq \frac{1}{2} \left( 1 + 2^{\mu} +
\frac{2^{\mu}}{\mu}\right )
\end{equation}

We then have our bound
\begin{equation}
  \|\nabla \nabla u \|_{L^{\infty}} \leq \left( 4 + \frac{1}{\mu}
  \right) \left( 1 + \log \left\lbrack  \frac{|\nabla \varphi |_{\mu}
L^{\mu}}{|\nabla \varphi |_{\inf}} \right\rbrack \right) 
\end{equation}
and use the estimates in Proposition \ref{prop:c1} to conclude.
\end{proof}
The proof of Corollary \ref{cor:com2} and the Geometric Lemma may be found in
the Appendix of \cite{Bertozzi1993GlobalRF}. It follows that
(\ref{eq:dynDW}) admits a Gr\"{o}nwall-type estimate.
\begin{prop}
  \label{prop:c2}
  Suppose $P_0$ bounded and that $\varphi_0 \in C^{2 + \mu} $
  in (\ref{eq:cauchyphiic}). Then, 
  %there exists a
  %constant $C_\mu$ depending only on  $ \mu, L, \lvert\W_0\rvert_{\mu}$ and $ |\W_0
  %|_{\inf}$, such that 
  the unique global
  solution $\varphi$ has $\nabla \W = \nabla \nabla^{\perp}\varphi$ which satisfies:
  \begin{equation*}
    |\nabla \W(t) |_{\mu} \leq
  \exp((C_3 + \mu) C_{L} \lvert t\rvert) \left \lbrack |  \nabla \W_0 |_{ \mu}  + C_4 \int_0^{t}
  (1 + |s|) |\W(s) |_{\mu} \, d s \right \rbrack,
% + C_3 \lvert\W_0\rvert_{\mu} (\lvert t\rvert + 1)
% \exp((C_1 + \mu)C_{L}
% \lvert t\rvert) ,
  \end{equation*}
  \begin{equation*}
    \|\nabla \W (t)\|_{L^{\infty}} \leq
 \exp(C_{L} \lvert t\rvert) \left \lbrack   \|\nabla \W_0\|_{L^{\infty}} + C_4 \int_0^{t} (1 +
  |s|) \|\W(s)\|_{L^{\infty}} \, ds \right\rbrack, 
 %+ C_3 \|\W_0\|_{L^{\infty}}(\lvert t\rvert + 1) \exp(C_{L}
 %\lvert t\rvert). 
  \end{equation*}
  for positive $t$.
\end{prop}
\begin{proof}[Proof of Theorem \ref{thm:reg} for $\partial P_0 \in C^{2 + \mu}$]
 With Proposition \ref{prop:c1} and \ref{prop:c2}, we see that
 \begin{equation}
  \|\W(t)\|_{C^{1 +\mu}} \leq C_5 \|\W_0\|_{C^{1+\mu}} (1+\lvert t\rvert)
 \exp(C_{L} \lvert t\rvert)
 \end{equation}
 for some constant $C_5$ depending only on $\mu$, $L$ and  $\varphi_0$, as desired.
\end{proof}

\section{Details of Figure \ref{fig:circ}}\label{sec:fig}
\subsection{Description of the numerical solver}\label{secA0}

In the compact domain $\T^2$, we develop a solver for the density patch problem
in the system (\ref{eq:Bs}) which resolves the dynamics of the
patch boundary $\partial P(t)$ in Corollary \ref{cor:patchtor} using a
level-set method that is second-order in time and first-order in
space (for background, see \cite{sethian1996level}).

The algorithm begins with the expression
\begin{equation}
u = -\nabla^{\perp} (\Delta^ 2)^{-1} \partial_1 \theta,
\end{equation}
where the scalar $\theta: \T^2 \to \R$ has zero mean, so the operator
$(\Delta^ 2)^{-1}$ on $\partial_1 \theta$ is well-defined and the vector-field $u: \T^2 \to \R^2$ has zero mean also.
The discrete Fourier coefficients $\widehat{u}$ are related explicitly
to $\widehat{\theta}$
\begin{equation}
\widehat{u}(k) = \frac{k^{\perp} k_1}{|k|^{4}} \widehat{\theta}(k)
\end{equation}
for nonzero $k \in  2\pi \Z^2$ and $k^{\perp} = (-k_2,k_1)$.
Accordingly, we set $\widehat{u}(0) = 0$. 

With the spacing $h = 1 / N$, we discretize the
space variable onto an $N \times N$ lattice with coordinates 
$x_{ij} = h (i - N/2,j - N / 2)$
for $i,j =0,\ldots,N-1$. 
Note that we
have fixed $N$ to some power of two, and write 
 $\varphi_{ij}(t) :=
\varphi(x_{ij},t)$ such that  $\theta_{ij} = H(\varphi_{ij})$. With this
convention, we have
\begin{equation}
\label{eq:fourtheta}
\widehat{\theta}(k) \approx h^2 \sum_{j=0}^{N-1} \sum_{i = 0}^{N-1}
H(\varphi_{ij}) \exp \left(-\iota k  \cdot x_{ij}  \right),
\end{equation}
where $\iota$ here is the imaginary unit. We thus resolve the
advecting velocity with spectral accuracy:
\begin{equation}
u_{ij} \approx \sum_{|k| \leq \pi N} 
\frac{k^{\perp} k_1}{|k|^{4}} \widehat{\theta}(k) \exp (\iota
k \cdot x_{ij}).
\end{equation}

%To evolve $\varphi_{ij}(t)$ we discretize time $t$ with index $n$. 
%That is
%$\varphi^{n}_{ij} := \varphi_{ij}(t^{n})$,  
%where
%$t^{n+1} = t^n + \Delta t^{n}$. 
The advection operator $F(\varphi) = - u \cdot \nabla \varphi$ is discretized  according to the monotone first-order upwinding scheme
\begin{equation}
F(\varphi_{ij}) \approx  (u_{ij})^{-} \cdot D^{+}_{ij}- (u_{ij})^{+}
 \cdot D^{-}_{ij},
\end{equation}
where the signed velocities $(v)^{\pm} = \max(\pm v,0) $ with $v = (v)^{+} - (v)^{-}$
product with the signed gradient 
\begin{equation}
\label{eq:pmD}
  D^{\pm}_{ij} = \pm\frac{1}{h} \begin{pmatrix} \varphi_{i\pm1,j} -
  \varphi_{ij} \\ \varphi_{i,j\pm1} - \varphi_{ij} \end{pmatrix}. 
\end{equation}

With our procedure given by (\ref{eq:fourtheta}-\ref{eq:pmD}), we integrate the system 
\begin{equation}
\frac{\partial\varphi_{ij}}{\partial t} = F(\varphi_{ij})
\end{equation}
in time using the second-order SSRK (Heun's) method where the
time step is chosen such that $CFL \leq 1 /2$, to complete the algorithm.

%Briefly, this involves the solving the advection problem
%(\ref{eq:cauchyphi}) via explicit first-order upwinding for the discretized
%level-set function $\varphi_{ij} := \varphi(x_{ij})$
%on a fixed $N \times N$ grid points $x_{ij}$ with uniform spacing $h = 1
%/ N$. 
%In the process, the advection velocity $u_{ij} := u(x_{ij})$ is found as
%a numerical solution to
%equation the momentum equation of ($B_{*}$) for  $\theta_{ij} = H(\varphi_{ij}) - \text{area}(P_0)$, via spectral
%collocation method. For time integration, we use Heun's method which is strong-stability preserving.  The full details of the
%algorithm and its verification are given in the
%Appendix.

%Since the method has a fixed mesh scale $h = 1 / N$ for an efficient
%spectral method, there is no
%local refinement around sections of $\partial P(t)$, which have high
%curvature growth. 
%As the curve sharpens aorund these, there is a breakdown of the
%algorithm's accuracy and we observe the rapid decay
%of $\text{area}(P(t))$ in our numerical solution 

%While the dynamics of $\partial P(t)$ do not depend on the details of
%the level set function $\varphi(t)$  away from its zero level-set, the quantities
%$\| \nabla^{\perp} \varphi(t) \|_{L^{\infty}}$  and $\|\nabla \nabla
%\varphi(t)\|_{L^{\infty}}$ very much do. Practically, this means we must
%choose $\varphi_0$ carefully. 

\subsection{Verification of convergence}\label{secA1}

Our verification of the numerical solver described in Section \ref{secA0}
addresses the implementation of the following three routines: 
\begin{enumerate}
\item\label{ver1} the level-set
method for the transport equation, 
\item\label{ver2} the spectral-collocation method for the momentum
equation, 
\item the coupling of \eqref{ver1} and \eqref{ver2} to approximate patch
solutions in the full system. 
\end{enumerate}

The implementation of first-order
upwind to transport the level-set function $\varphi$ was tested against various
fixed divergence-free velocity fields on $\T^{2}$: the shear  $u(x) =
(x_2^{3},0)$, and cellular flows $u = \nabla^{\perp}\psi$ where 
\begin{equation}
\psi(x) =
\sin(k_{1} x_{1}) \sin(k_{2} x_{2}),
\end{equation}
for various mode numbers $k \in 2 \pi \Z^2$. 

The
spectral solver for the momentum equation was tested on sinusoidal
temperature distributions
\begin{equation}
\theta(x) = \cos(k_{1}x_{1} + k_{2} x_2)
\end{equation}
and it indeed recovers the exact solution
\begin{equation}
u(x) = \frac{k_{1}k^{\perp}}{|k|^{4}}
\cos(k_1 x_1 + k_2 x_2)
\end{equation}
up to machine precision for any $|k| \leq \pi N$, as expected by the
Nyquist-Shannon sampling theorem. 

These two solvers are coupled together
in the full algorithm, so we verified that quantities like $\|\nabla
\nabla u\|_{L^{\infty}}$ and $\|\nabla\nabla\varphi\|_{L^{\infty}}$
from simulations with the same initial data but 
various $N$ converge for short time we increase $N$.   

Finally, we examine if the numerical solver is appropriately handling the
dynamics of the  low regularity solution which has initial data $\varphi_0(x) =
\Phi(x)$. Consider the regularized system
\begin{equation}
\label{eq:Ps}
\tag{$\Phi_{\epsilon}$}
\left \{
  \begin{aligned}
    \left( \frac{\partial}{\partial t} + u\cdot\nabla\right) \varphi &= 0, \\
    -\Delta u +
    \nabla \Pi  &= \phi_{\epsilon} *H(\varphi) e_2, \\
    \mathrm{div}\, u &= 0,
  \end{aligned}
\right.
\end{equation}
where we have $\epsilon > 0$ and  the mollifier is the Gaussian
\begin{equation}
\phi_{\epsilon }(z) = \frac{1}{\epsilon\sqrt{2 \pi} } \exp\left(-
  \frac{z^2}{2\epsilon^2}\right).
\end{equation}

The system (\ref{eq:Ps}) is solved with the algorithm given above for
system (\ref{eq:Bs}), except the Gaussian filter  with standard deviation  $\epsilon$
(convolution with $\phi_{\epsilon}$) is applied numerically to the
points $H(\varphi_{ij})$ before the discrete Fourier transform in
(\ref{eq:fourtheta}) is taken. We observe that the tendency towards curvature singularities
observed in Figure \ref{fig:circ} is suppressed in this regularized system for any
fixed $\epsilon$. However, as we take epsilon to machine precision,
the simulations recover the results of $\epsilon = 0$  (system (\ref{eq:Bs})), in particular the picture in Figure \ref{fig:circ}.

\subsection{The simulation in Figure \ref{fig:circ}}\label{secA2}
The curve illustrated in the Figure is the zero level set of $\varphi_{ij}(t)$
whose initial condition for the algorithm was specified as
$\varphi_0(x_{ij}) = \Phi(x_{ij})$, where 
\begin{equation}
\label{eq:circphi}
\Phi(x) = \cos(2 \pi | x - \xi|), \quad \xi =
\left(\frac{1}{2},\frac{1}{2}\right) 
\end{equation}
has precisely the circle of radius one-half
$\Sone(\frac{1}{2})$ centered at $\xi$ as its zero level set. The patch
moves upwards, as does the curve $\partial P(t)$. To depict the
changing shape of $\partial P$ in time, the frame where the curve is
drawn follows this
movement.

While patch solutions to
our problem maintain their area as they evolve in time, the sharpening
of the $\partial P(t)$ in the fixed-grid simulations results in
rapid decreases in the area once the variations of the curve are of the scale $h =
1/N$. Thus for
the resolution given by $N = 1024$, we terminate the patch simulation at $t =
100$, before the relative error of the patch area is more than $0.01$.

\section{Acknowledgments}
I thank my advisor, Peter Constantin, for suggesting this problem to me,
and for
his guidance thereof. I also express my gratitude to Greg Hammett, for his
guidance on developing the simulation. The author thanks the anonymous reviewer \#2 for their helpful comments. During this project, the author was partially supported
by the Ford Foundation and the NSF Graduate Research Fellowship grant DGE-2039656.
\bibliographystyle{nar}
\bibliography{local}

%%===========================================================================================%%
%% If you are submitting to one of the Nature Portfolio journals, using the eJP submission   %%
%% system, please include the references within the manuscript file itself. You may do this  %%
%% by copying the reference list from your .bbl file, paste it into the main manuscript .tex %%
%% file, and delete the associated \verb+\bibliography+ commands.                            %%
%%===========================================================================================%%

%\bibliography{sn-bibliography}% common bib file
%% if required, the content of .bbl file can be included here once bbl is generated
%%\input sn-article.bbl

%% Default %%
%%\input sn-sample-bib.tex%

\end{document}